\documentclass[11pt,leqno]{article}

\usepackage{amsthm,amsfonts,amssymb,amsmath,color}
\usepackage{amsbsy}
\usepackage{showkeys}
\numberwithin{equation}{section}

\newcommand\R{\mathbb R}

\def\de{\mathfrak{z}}

\def\O{\Omega}

\def\l{\lambda}

\def\epsilon{\varepsilon}
\def\e{\varepsilon}

\newcommand\br{\begin{rem}}
\newcommand\er{\end{rem}}
\newcommand\bp{\begin{pmatrix}}
\newcommand\ep{\end{pmatrix}}
\newcommand\be{\begin{equation}}
\newcommand\ee{\end{equation}}
\newcommand\ba{\begin{equation}\begin{aligned}}
\newcommand\ea{\end{aligned}\end{equation}}

\newcommand{\pder}[2] {\frac{\partial #1}{\partial #2}}

\usepackage{layout}

\renewcommand{\R}{{\mathbb R}}

\newcommand{\TT}{{\mathbb T}}

\newcommand{\II}{{\mathbb I}}

\newcommand{\SSS}{{\mathbb S}}
\newcommand{\tn}[1]{{\mathbb #1}}

\newcommand{\tvr}{\widetilde \varrho}
\newcommand{\tvu}{\widetilde \vu}
\newcommand{\tnT}{\widetilde{\tn{T}}}
\newcommand{\teta}{\widetilde \eta}

\newcommand{\ff}{{\mathbf f}}

\newcommand{\vr}{\varrho}

\newcommand{\vu}{\vc{u}}

\newcommand{\vcg}[1]{{\boldsymbol #1}}
\newcommand{\vc}[1]{{\bf #1}}
\newcommand{\Div}{{\rm div}\,}

\newcommand{\dx}{\,{\rm d} {x}}
\newcommand{\dt}{\,{\rm d} t }
\newcommand{\dtau}{\,{\rm d} \tau }

\newtheorem{definition}{Definition}[section]
\newtheorem{theorem}[definition]{Theorem}

\newtheorem{lemma}[definition]{Lemma}

\newtheorem{remark}[definition]{Remark}

\begin{document}

\title{Weak solutions and weak-strong uniqueness for a compressible power-law-Oldroyd--B fluid model}

\author{Yong Lu \footnote{Department of Mathematics, Nanjing University, 210093 Nanjing, China. Email: {\tt luyong@nju.edu.cn}.} \and  Milan Pokorny\footnote{Mathematical Institute of Charles University, Faculty of Mathematics and Physics, Charles University, Sokolovsk\'a 86, 186 75,   Prague. Email: {\tt pokorny@karlin.mff.cuni.cz}.}  \thanks{}}

\date{}

\maketitle

\begin{abstract}
We consider a model of a viscoelastic compressible flow which is additionally shear thickening (the stress tensor corresponds to the power law model, however, the divergence of the velocity is due to the model bounded). We prove existence of a weak solution to this model provided the growth in the power law model is larger or equal than $\frac 52$. We also prove that any sufficiently smooth solution of this model is unique in the class of weak solution, provided extra integrability of the initial value for the extra stress tensor is assumed. However, we are not able to prove in general existence of such a strong solution, even not locally in time. 

\end{abstract}





\section{Introduction}
\label{s_1}

The visco-elastic properties of fluids are an important phenomenon which can be observed for many fluids, including also certain regimes for the blood flow. One of prominent examples are the dilute polymeric fluids. While in the past, the major attention was oriented towards incompressible fluids (\cite{Renardy90, F-G-O02}), in the last decade more and more attention is also paid to compressible visco-elastic fluids. It is connected with better understanding of the mathematical fluid mechanics for compressible fluids as well as with the more and more complex fluids used in applications, where often the compressibility of the fluid simply cannot be neglected.

One of the prominent models for modelling incompressible visco-elastic fluids is the Oldroyd-B model which is widely used to describe the flow of dilute polymeric fluids. From the incompressible Navier--Stokes--Fokker--Planck system which is a micro-macro model  one can derive (formally) the incompressible Oldroyd--B model in dumbbell hookean setting, see \cite{Bris-Lelievre}. For the compressible fluids, a similar derivation was performed in \cite{Barrett-Lu-Suli}, where a compressible Oldroyd--B model was derived as a macroscopic closure of the compressible Navier--Stokes--Fokker--Planck equations; the latter was intensively studied in a series of papers by Barrett and S\"uli \cite{Barrett-Suli2, Barrett-Suli4, Barrett-Suli1, Barrett-Suli, BS2016}.

The derived model in \cite{Barrett-Lu-Suli} is described as follows. Let $\Omega \subset \mathbb{R}^3$ be a bounded open domain with a smooth boundary. We denote $Q_T:=(0,T)\times \O$ and have
\begin{alignat}{2}
\label{01a}
\pder{\vr}{t} + \Div (\vr \vu) &= 0,
\\
\label{02a}
\pder{(\vr\vu)}{t}+ \Div (\vr \vu \otimes \vu) +\nabla p(\vr)  - \Div \SSS(\nabla \vu) &=\Div \big(\TT - (kL\eta + \de\, \eta^2)\,\II\, \big)  +  \vr\, \ff,
\\
\label{03a}
\pder{\eta}{t} + \Div (\eta \vu) &= \e \Delta \eta,
\\
\label{04a}
\pder{\TT}{t} + {\rm Div} (\vu\,\TT) - \left(\nabla \vu \,\TT + \TT\, \nabla^{\rm T} \vu \right) &= \e \Delta \TT + \frac{k}{2\lambda}\eta  \,\II - \frac{1}{2\lambda} \TT,
\end{alignat}
where $\SSS(\nabla \vu)$ is the {\em Newtonian stress tensor} defined by
\be\label{Newtonian-tensor}
\SSS(\nabla \vu) = \mu^S \left( \frac{\nabla \vu + \nabla^{\rm T} \vu}{2} - \frac{1}{3} (\Div \vu) \II \right) + \mu^B (\Div \vu) \II,
\ee
and $\{{\rm Div}\, (\vu \tn{T})\}_{ij} = \sum_{k=1}^3 \pder{(u_kT_{ij})}{x_k}$.
Above, the constants $\mu^S>0$ and $\mu^B\geq 0$ are the shear and bulk viscosity coefficients, respectively. The pressure $p$ is a function of the density $\vr$ and in \cite{Barrett-Lu-Suli} $p(\vr)$  was assumed in the form
\be\label{pressure_old}
 p(\vr)=a \vr^\gamma, \quad a>0, \ \gamma \geq 1.
\ee

The extra stress tensor $\TT = T_{ij}$, $1\leq i,j \leq 3$ is a positive definite symmetric matrix defined on $(0,T)\times \O$.
%
%

The polymer number density $\eta$ is a nonnegative scalar function defined as the integral of the probability density function $\psi$, governed by the Fokker--Planck equation, for the conformation vector, which is a microscopic variable in the modeling of dilute polymer chains. The term $q(\eta):= kL \eta + \de \eta^2$ in the momentum equation \eqref{02a} can be seen as the \emph{polymer pressure}, compared to the fluid pressure $p(\vr)$.

The meanings of the various quantities and parameters appearing in \eqref{01a}--\eqref{04a} were introduced in the derivation of the model in \cite{Barrett-Lu-Suli}.  In particular, the parameters $\e$, $k$, $\l$, $\de$, $L$ are non-negative. The problem is accompanied with suitable boundary conditions (homogeneous Dirichlet condition for the velocity and homogeneous Neumann condition for the extra stress and the polymer density $\eta$) and initial conditions for the density, linear momentum, extra stress tensor and the polymer density. 

In the same paper, the existence of weak solutions for the two space dimensional problem was shown. One of the main points of the proof was the fact that in the equation for the extra stress tensor, the extra stress diffusion term is present. Another important point was the fact that by suitable approximation, it is possible to verify that if the extra stress tensor is initially symmetric and positive definite, it remains so for all times. These points lead to suitable estimates and using the fact that the problem is studied in two space dimensions, the nonlinear terms of the type $\nabla \vu \TT$ can be controlled by the Gronwall argument.

The other known results are small data ones. The existence and uniqueness of local strong solutions and the existence of global solutions near an equilibrium for  three-dimensional compressible visco-elastic fluids was considered in \cite{Qian-Zhang, Hu-Wang3, Hu-Wu, Lei}. In particular, Fang and Zi \cite{Fang-Zi} not only proved the existence of a unique local-in-time strong solution, but also established a blow-up criterion for strong solutions. Later, in \cite{Lu-Z.Zhang18}, one of the authors and his collaborator showed the weak-strong uniqueness, presented a refined blow-up criterion and proved a conditional regularity result in two dimensional setting. 

Inspired by the work of Lions and Masmoudi \cite{LM00}, considering the corotational derivative setting and assuming the extra stress tensor to be a scalar matrix, it is possible to obtain a simplified model which takes a similar form as the multi-component compressible Navier-Stokes equations. Based on the approach from \cite{VWY19, Novotny-Pokorny-19} it is possible to show the well posedness of the problem in the weak setting even if the extra stress tensor diffusion is omitted, see \cite{LuPo1}. A similar problem in the setting of strong solutions (i.e. data close to an equilibrium state), was considered in \cite{Liu-Lu-Wen-21}.

Before coming to our problem, let us recall that recently, several new results for the incompressible visco-elastic fluid models were achieved, see e.g. \cite{Bathory-Bulicek-Malek-21}, \cite{Bathory-Bulicek-Malek-24} or \cite{Bulicek-Malek-Rodriguez-22}, but in three space dimensions, the global-in-time well-posedness of the Oldroyd-B model for large data is still an open questions and certain regularizations (however, physically justified) need to be included.

The aim of this paper is to consider the Oldroyd-B compressible model in three space dimensions. Indeed, we cannot do it directly; we partially follow an interesting attempt from paper \cite{Kreml-Pokorny-Salom-15} where in the incompressible setting, the authors considered combination of power-law stress tensor and nonlinear ($q$-Laplacian) extra stress tensor diffusion to get weak well-posedness of the problem for large data. Since no extra estimates coming from other terms or from the positive definiteness of the extra stress tensor were considered, the extra stress tensor diffusion was always superlinear (the power $q>2$). In this paper we improve the situation and consider only linear extra stress diffusion which is possible to justify physically. 

On the other hand, there is a serious problem with compressible power-law fluid model and the well-posedness (in the weak setting for large data and large time interval) is known only in case of exponential growth of the stress tensor (see \cite{Mamontov1, Mamontov2, Vodak}). The polynomial growth was considered in \cite{Zhikov-Pastuchova}, however, the proof contains a serious gap (at certain moment the author used the strong convergence of the velocity sequence which is not clear unless the density is strictly positive; this, however, is not evident how to show). Therefore, inspired by the considerations in \cite{Feireisl-Liu-Malek} we adopt their choice of the stress tensor which directly implies bounded velocity divergence and thus improves the properties of the solutions and removes the gap in the proof from \cite{Zhikov-Pastuchova}.

To conclude the introductory part, let us formulate our problem we plan to study. We consider system \eqref{01a}--\eqref{04a}, however, in a different setting than in \cite{Barrett-Lu-Suli}. The viscous stress tensor $\tn{S}$ is assumed in the form
\begin{equation} \label{viscous_stress}
\tn{S}(\nabla \vu) = 2\mu_0 (1+ |\tn{D}^d(\vu)|^2)^{\frac {r-2}{2}} \tn{D}^d(\vu) + \beta (\Div\vu) \tn{I} 
\end{equation}  
with $\tn{D}^d(\vu)$ the deviatoric part of the velocity gradient,
\begin{equation} \label{deviatoric}
\tn{D}^d(\vu) = \tn{D} (\vu) -\frac 13 \Div \vu \tn{I} = \frac 12 \Big(\nabla \vu + \nabla \vu^T -\frac 23 \Div \vu \tn{I} \Big),
\end{equation}
the exponent $r\geq \frac 52$, $\mu_0 >0$ and the function $\beta$: $(-\frac 1b,\frac 1b)\to [0,\infty)$ such that there exists a convex potential $\Lambda$: $\tn{R} \to [0,\infty]$ with the following properties
\begin{equation} \label{potential}
\begin{aligned}
\Lambda(0)=0, & \qquad \Lambda'(z) = z \beta(z) \\
\Lambda(z) \to +\infty \text{ if } z\to \pm \frac 1b, & \qquad \Lambda(z) = +\infty \text{ if } |z|\geq \frac 1b.
\end{aligned}
\end{equation}
The pressure $p$ is defined as
\begin{equation} \label{pressure}
p(\vr) = \vr^2 \psi'(\vr),
\end{equation}
where $p\in C([0,\infty)) \cap C^2((0,\infty))$, $p(0)=0$, $p'(\vr)>0$ for $\vr>0$ (thus the Helmholtz free energy $\psi \in C^3(0,\infty)$, $2\psi' + \vr \psi'' >0$ for $\vr>0$ and $\lim_{\vr\to 0^+} \vr^2\psi'(\vr) =0$).

Finally, the coefficients $k$, $\zeta$, $L$, $\lambda$ and $\varepsilon$ are positive. Since their physical meaning is not important in our paper, we refer to \cite{Barrett-Lu-Suli} for their physical interpretation and further details.
 
The system is accompanied with the initial conditions
\begin{equation} \label{initial}
\begin{aligned}
\vr(0,x) = \vr_0(x),  \quad (\vr\vu)(0,x) = \vc{m}_0(x)=(\vr_0 \vu_0)(x), \quad 
\eta(0,x) =\eta_0(x),  \quad \tn{T}(0,x) = \tn{T}_0(x),
\end{aligned}
\end{equation}
and the boundary conditions at $(0,T)\times \partial \Omega$
\begin{equation} \label{boundary}
\vu = \vc{0}, \qquad \pder{\tn{T}}{\vc{n}} = \tn{O}, \qquad \pder{\eta}{\vc{n}} =0,
\end{equation}
where $\vc{n}$ is the outer normal vector to $\partial \Omega$. 

Note that we could replace the homogeneous Dirichlet boundary conditions for the velocity by the Navier ones (with no inflow/outflow) or consider the space-periodic boundary problem and we would get the same results as presented below.

In what follows, we first introduce the weak formulation of our problem and state the main result. Section \ref{s_3} contains a priori estimates to justify the choice of the exponent $r \geq \frac 52$. Section \ref{s_4} will be devoted to the construction of weak solutions to our problem and the last Section \ref{s_5} will present the proof of the second main result, the weak-strong uniqueness property.

Below, we use standard notation for the Lebesgue spaces $L^p(\Omega)$, Sobolev spaces $W^{1,q}(\Omega)$ and Bochner spaces $L^r(0,T;X)$ with $X$ a Banach space, endowed with the standard norms. The positive constants are denoted by $C$ and their value may change even in the same formula or on the same line.

\section{Weak solution, main results}
\label{s_2}

We first introduce the weak formulation of our problem and then state the main results of this paper.

\subsection{Weak formulation}
\label{s_2.1}

We consider the so-called renormalized finite energy weak solutions which are weak solutions satisfying the renormalized continuity equation (which is for free in our setting) and the energy inequality in the integral form.

\begin{definition} \label{weak_solution}
We say that the quadruple $(\vr,\vu,\eta,\tn{T})$ is a weak solution to our problem \eqref{01a}--\eqref{04a}, \eqref{viscous_stress}--\eqref{boundary} provided
\begin{equation} \label{integrability}
\begin{aligned}
\vr \in & L^\infty((0,T)\times \Omega) \cap C([0,T]; L^1(\Omega)), \qquad \vr >0 \quad \text{ a.e. in } (0,T)\times \Omega \\
\vu \in & L^\infty(0,T;L^2(\Omega;\R^3)) \cap L^r(0,T; W^{1,r}_0(\Omega;\R^3)), \qquad \vr\vu \in C([0,T];L^1(\Omega;\R^3)) \\
\Div \vu &\in \Big(-\frac 1b,\frac 1b\Big) \quad \text{ a.e. in } (0,T)\times \Omega \\  
\eta \in & L^\infty(0,T;L^2(\Omega)) \cap L^2(0,T;W^{1,2}(\Omega)) \cap C([0,T];L^1(\Omega)) \\
\tn{T} \in & L^\infty(0,T;L^2(\Omega;\R^{3\times 3})) \cap L^2(0,T;W^{1,2}(\Omega;\R^{3\times 3})) \cap C([0,T];L^1(\Omega;\R^{3\times 3})), 
\end{aligned}
\end{equation}
the tensor $\tn{T}$ is symmetric and positive definite, and the following identities hold
\begin{equation} \label{weak_continuity}
\int_0^t \int_{\R^3} \Big(\vr \pder{\psi}{t} + \vr\vu\cdot \nabla \psi\Big) \dx\dtau = \int_{\R^3} (\vr \psi)(t,x)\dx - \int_{\R^3}\vr_0(x) \psi (0,x) \dx 
\end{equation}
for any $t \in (0,T]$, for any $\psi \in C^\infty_0([0,T]\times \R^3)$, where $\vu$ and $\vr$ are extended by zero outside of $\Omega$;
\begin{equation} \label{weak_momentum}
\begin{aligned} 
&\int_0^t \int_\Omega \big(\Lambda (\Div \vcg{\varphi})-\Lambda(\Div \vu)\big) \dx \dtau \geq \frac 12 \int_\Omega \big((\vr|\vu|^2) (t,x) -(\vr_0 |\vu_0|^2)(x) \big)\dx \\
&-\int_\Omega (\vr\vu\cdot \vcg{\varphi})(t,x) \dx + \int_\Omega \vc{m}_0(x) \cdot\vcg{\varphi}(0,x)\dx + \int_0^t \int_\Omega \Big(\vr \vu \cdot \pder{\vcg{\varphi}}{t} + \vr(\vu\otimes \vu):\nabla \vcg{\varphi} \Big)\dx\dtau \\
&+ \int_0^t \int_{\Omega} \big(2\mu_0 (1+ |\tn{D}^d(\vu)|^2)^{\frac{r-2}{2}} : \tn{D}^d(\vu-\vcg{\varphi}) + p(\vr) \Div (\vcg{\varphi} -\vu)\big) \dx\dtau \\
&+ \int_0^t \int_\Omega  \big(\vr\vc{f} \cdot (\vcg{\varphi} - \vu) - \tn{T}: \nabla (\vcg{\varphi}-\vu) + (kL\eta + \zeta\eta^2) \Div (\vcg{\varphi}-\vu)\big)\dx\dtau  
\end{aligned}
\end{equation}
for a.e. $t\in (0,T]$ and for any $\vcg{\varphi} \in C^\infty_0([0,T]\times \Omega;\R^3)$;
\begin{equation} \label{weak_eta} 
\int_0^t \int_\Omega \Big(\eta \pder{\psi}{t} + \eta \vu\cdot \nabla \psi -\varepsilon \nabla \eta \cdot \nabla \psi\Big) \dx\dtau = \int_\Omega \big((\eta \psi)(t,x) -\eta_0(x)\psi(0,x)\big)\dx
\end{equation}
for any $t\in (0,T]$ and for any $\psi \in C^\infty([0,T]\times \overline{\Omega})$;
\begin{equation} \label{weak_extra_stress}
\begin{aligned}
&\int_0^t \int_\Omega \Big(\tn{T}:\pder{\tn{M}}{t} + \vu \tn{T}:: \nabla \tn{M} \Big)\dx\dtau + \int_0^t \int_\Omega \big(\nabla \vu \tn{T} + \tn{T}(\nabla \vu)^T\big) :\tn{M}\dx\dtau \\
&- \int_0^t \int_\Omega \Big(\varepsilon \nabla \tn{T}::\nabla \tn{M} +\frac 1{2\lambda} \tn{T}:\tn{M} \Big)\dx\dtau + \frac{k}{2\lambda}\int_0^t \int_\Omega \eta {\rm Tr}\, \tn{M} \dx\dtau \\
&= \int_\Omega \big((\tn{T}:\tn{M}) (t,x) - \tn{T}_0 (x) :\tn{M}(0,x)\big)\dx
\end{aligned}
\end{equation}
for any $t\in (0,T]$ and for any $\tn{M} \in C^\infty([0,T]\times \overline{\Omega};\R^{3\times 3})$, the operator ${\rm Tr}\,\tn{M} = \sum_{i=1}^3 M_{ii}$ is the trace operator in the matrix sense and $\vu \tn{T}:: \nabla \tn{M} = \sum_{k=1}^3 u_i T_{jk}\pder{M_{jk}}{x_i}$.

If, moreover, the renormalized continuity equation 
\begin{equation} \label{renor_cont}
\begin{aligned}
&\int_0^t \int_{\R^3} \Big(b(\vr) \pder{\psi}{t} + b(\vr)\vu\cdot \nabla \psi\Big) \dx\dtau + \int_0^t\int_\Omega (b(\vr)-b'(\vr)\vr) \Div \vu \dx\dtau \\ 
&= \int_{\R^3} (b(\vr) \psi)(t,x)\dx - \int_{\R^3}b(\vr_0)(x) \psi (0,x) \dx 
\end{aligned}
\end{equation}
holds for any $t \in (0,T]$, for any $\psi \in C^\infty_0([0,T]\times \R^3)$, where $\vu$ and $\vr$ are extended by zero outside of $\Omega$ and for any $b$: $[0,\infty)\to \R$, where $b\in C^\infty([0,\infty))$ and $b'(z) =0$ for all $z\geq M$, where $M>0$ is some finite number, the solution is called a weak renormalized solution.

Finally, the solution is called a weak (renormalized) finite energy solution, provided it is a weak (renormalized) solution and it fulfils the energy inequality
\begin{equation} \label{energy_inequality}
\begin{aligned}
&\int_\Omega \Big( \frac 12 \vr|\vu|^2 + P(\vr) +kL\eta \ln \eta +\zeta \eta^2 + \frac 12 {\rm Tr}\, \tn{T} \Big) (t,x) \dx  \\
& + 2\varepsilon \int_0^t \int_\Omega \Big(2kL |\nabla \eta^{\frac 12}|^2 + \zeta |\nabla \eta|^2 + \frac 1{4\lambda} {\rm Tr}\, \tn{T} \Big)\dx\dtau \\
& + \int_0^t \int_\Omega \Big(2\mu_0 (1+ |\tn{D}^d(\vu)|^2)^{\frac {r-2}{2}} |\tn{D}^d(\vu)|^2 + \Lambda (\Div \vu)\Big)\dx\dtau \\
&\leq  \int_\Omega \Big( \frac 12 \vr_0|\vu_0|^2 + P(\vr_0) +kL\eta_0 \ln \eta_0 +\zeta \eta_0^2 + \frac 12 {\rm Tr}\, \tn{T}_0 \Big) (x) \dx  \\
& + \int_0^t \int_\Omega \vr \vc{f}\cdot \vu \dx \dtau + \frac{3k}{4\lambda}\int_0^t \int_\Omega \eta \dx\dtau
\end{aligned}
\end{equation}
for a.e. $t \in (0,T]$; above, $P(\vr) = \vr \psi(\vr)$ is the pressure potential.
\end{definition}  

\subsection{Main results}
\label{s_2.2}

We now formulate the main results of this paper. The first one concerns the weak solvability of our problem.
\begin{theorem} \label{main_1}
Let $\Omega$ be a bounded domain of the class $C^{2,\beta}$ for some $0<\beta \leq 1$. Let $\vr_0 \in L^\infty(\Omega)$, $0<C_0 \leq \vr_0$ a.e. in $\Omega$, $\frac{\vc{m}_0}{\vr_0} \in L^2(\Omega;\R^3)$, $\tn{T}_0 \in L^2(\Omega;\R^{3\times 3})$ be symmetric positive definite, and  $\eta_0\in L^2(\Omega)$ be non-negative a.e. Let $\vc{f}\in L^\infty((0,T)\times \Omega;\R^3)$, $\varepsilon$, $k$, $L$, $\zeta$, $\lambda$ be positive constants and let $r\geq \frac 52$. 

Then there exists a renormalized weak finite energy weak solution to our problem  \eqref{01a}--\eqref{04a}, \eqref{viscous_stress}--\eqref{boundary} in the sense of Definition \ref{weak_solution}.
\end{theorem}

The other result deals with the weak-strong uniqueness.

\begin{theorem} \label{main_2}
Let $(\vr,\vu,\tn{T},\eta)$ be a weak solution to our problem \eqref{01a}--\eqref{04a}, \eqref{viscous_stress}--\eqref{boundary} in the sense of Definition \ref{weak_solution} and let $(\widetilde\vr,\widetilde \vu,\widetilde {\tn{T}},\widetilde \eta)$ be a strong solution to the same problem corresponding to the same data. Let $\tvu$, $\nabla \tvu$, $\pder{\tvu}{t}$ be bounded, $\tnT \in L^2(0,T;L^\infty(\Omega;\R^{3\times 3}))$, $\nabla |\tn{D}^d(\tvu)|$, $\nabla \Lambda'(\Div \tvu)$, $\nabla \teta(1+\teta)$ and $\Div \tnT \in L^2(0,T;L^3(\Omega;\R^3))$. Let in addition to assumptions of Theorem \ref{main_1} the initial value $\eta_0 \geq C_2 >0$ and for $r\in [\frac 52,10]$ also $\tn{T}_0 \in L^{3+a}(\Omega;\R^{3\times 3})$ for some $a>0$. Then 
$$
(\vr,\vu,\tn{T},\eta) = (\widetilde\vr,\widetilde \vu,\widetilde {\tn{T}},\widetilde \eta)
$$
in the set $[0,T] \times \Omega$.
\end{theorem}

\section{A priori estimates}
\label{s_3}

Before showing the formal a priori estimates, we recall several general results which will be useful throughout the paper. 

\subsection{Preliminaries}
\label{s_3.1}

\begin{lemma} \label{Korn}
Let $\vc{v} \in W^{1,p}_0(\Omega;\R^3)$, $1<p<\infty$. Then there exists a constant $C=C(p)$ such that
\begin{equation} \label{K1}
\|\vc{v}\|_{1,p} \leq C \|\tn{D}^d(\vc{v})\|_p = C \|\nabla \vc{v} + \nabla^T \vc{v} -\tfrac 23 \Div \vc{v}\tn{I}\|_p.
\end{equation}
\end{lemma}

\begin{proof}
It is a direct consequence of \cite[Theorem 10.16]{FN_book}, since we may extend the function by zero to the whole $\R^3$.
\end{proof}

\begin{remark} \label{r1}
(i) The same result ($\frac 23$ replaced by $\frac 2N$) holds for any $N>2$. \newline
(ii) If $\vc{v} \in W^{1,p}(\Omega;\R^3)$ only (the case of the Navier boundary conditions, where we additionally have $\vc{v}\cdot \vc{n}=0$ on $\partial \Omega$), then we may use \cite[Theorem 10.17]{FN_book} to get
\begin{equation} \label{K2}
\|\vc{v}\|_{1,p} \leq C \Big(\|\tn{D}^d(\vc{v})\|_p + \int_\Omega \vr|\vc{v}|\dx\Big)
\end{equation}
which holds for any $\vc{v} \in W^{1,p}(\Omega;\R^3)$, $1<p<\infty$ and arbitrary $\vr \geq 0$ a.e. such that $\vr \in L^q(\Omega)$ for some $q>1$ such that $\int_\Omega \vr = M_0>0$. This is the main difference in the case of the Navier boundary conditions for the velocity. 
\end{remark}

Let us now consider the linear parabolic problem with the Neumann boundary conditions. We have for the problem
\begin{equation} \label{parabolic_Neumann}
\begin{aligned}
\pder{z}{t} -\varepsilon \Delta z = h & \qquad \text{in } (0,T)\times \Omega \\
z(0,x) = z_0(x) & \qquad \text{in } \Omega \\
\pder{z}{\vc{n}}=0 & \qquad \text{on } (0,T)\times \partial \Omega
\end{aligned}
\end{equation}
the following result
\begin{lemma} \label{estimates_parabolic_Neumann}
Let $\Omega \in C^{2,\Theta}$ for some $0<\Theta\leq 1$, $1<p,q<\infty$. Let the initial condition $z_0 \in \overline{\{z\in C^\infty(\overline{\Omega}: \pder{z}{\vc{n}} = 0 \text{ on } \partial\Omega\}}^{\|\cdot\|_{W^{2-\frac 2p,q}(\Omega)}}$ and let $h \in L^p(0,T;L^q(\Omega))$. Then there exists unique strong solution to problem \eqref{parabolic_Neumann} and it satisfies the following estimates
\begin{equation} \label{est_par_Neu_1}
\begin{aligned}
\varepsilon^{1-\frac 1p} &\|z\|_{L^\infty(0,T; W^{2-\frac 2p,q}(\Omega))} + \Big\|\pder{z}{t}\Big\|_{L^p(0,T;L^q(\Omega))} + \varepsilon \|z\|_{L^p(0,T;W^{2,q}(\Omega))} \\ 
& \leq C (\varepsilon^{1-\frac 1p} \|z_0\|_{W^{2-2p,q}(\Omega)} + \|h\|_{L^p(0,T;L^q(\Omega))}).
\end{aligned}
\end{equation}
If in particular $h=\Div \vc{g}$, $\vc{g} \in L^p(0,T;L^q(\Omega;\R^3))$, then we have
\begin{equation} \label{est_par_Neu_2}
\begin{aligned} 
\varepsilon^{1-\frac 1p} &\|z\|_{L^\infty(0,T;L^{q}(\Omega))} + \varepsilon \|\nabla z\|_{L^p(0,T;L^{q}(\Omega))} \\
&\leq C ( \varepsilon^{1-\frac 1p} \|z_0\|_{L^q(\Omega)} + \|\vc{g}\|_{L^p(0,T;L^q(\Omega))}).
\end{aligned}
\end{equation}
\end{lemma}

\begin{proof}
The proof can be compiled from results in \cite{Amann1}.
\end{proof}

Let the matrix $\tn{P}$ be real and symmetric. Then we can define the function calculus for matrices in the following way. There exists a matrix $\tn{Q}$ orthogonal (${\rm det}\, \tn{Q} =1$) such that $\tn{P} = \tn{Q} \tn{D} \tn{Q}^T$, where $\tn{D}$ is a diagonal matrix with the (real) eigenvalues of $\tn{P}$ on the diagonal, $\tn{D} = \{\lambda_i \delta_{ij}\}_{i,j=1}^3$. 

For a function $g$: $\R \to \R$ we may define
$$
g(\tn{P}) := \tn{Q} g(\tn{D}) \tn{Q}^T \qquad \text{ and } \qquad g'(\tn{P}) := \tn{Q} g'(\tn{D}) \tn{Q}^T,
$$
where 
$$
g(\tn{D}) = \{g(\lambda_i) \delta_{ij}\}_{i,j=1}^3 \qquad \text{ and } \qquad g'(\tn{D}) = \{g'(\lambda_i) \delta_{ij}\}_{i,j=1}^3.
$$
%

\subsection{Formal a priori estimates}
\label{s_3.2}

We deduce here formal a priori estimates under the assumption that a) the solution is as smooth as required in the procedure of obtaining the estimates b) the extra stress tensor is symmetric and positive definite in $(0,T)\times \Omega$. While the former will be later obtained for an approximate system which allows for smooth solutions, the latter will be justified by suitable approximation due to the same assumption on the initial datum $\tn{T}_0$. The main goal of this short subsection is to explain why we need the exponent $r\geq \frac 52$. In the construction of solutions this fact may be hidden by the number of special parameters and extra terms which allow for proving existence of an approximate solutions.

We multiply \eqref{02a} on $\vc{u}$ and integrate the resulted equality over $\Omega$ (or, equivalently use $\vcg{\varphi} = \vc{0}$ in \eqref{weak_momentum}). Applying also \eqref{01a} in the renormalized form
we end up with
\begin{equation} \label{as.1}
\begin{aligned}
& \frac{{\rm d}}{{\rm d}t} \int_\Omega \Big( \frac 12 \vr |\vu|^2 + P(\vr)\Big)\dx + \int_\Omega \Big(2\mu_0 (1+ |\tn{D}^d(\vu)|^2)^{\frac{r-2}{2}} |\tn{D}^d(\vu)|^2 + \Lambda (\Div \vu)\Big)\dx \\
&= \int_\Omega \vr\vc{f}\cdot \vu \dx - \int_\Omega \tn{T}:\tn{D}(\vu) \dx + \int_\Omega (kL\eta + \zeta \eta^2)\Div \vu \dx,  
\end{aligned}
\end{equation}
where $P(\vr) = \vr \int_0^\vr \frac{p(s)}{s^2}\,{\rm d}s = \psi(\vr)\vr$. Using the renormalized form of the equation for $\eta$ (alternatively, multiplying \eqref{03a} on $kL (\ln \eta +1) + 2\zeta\eta$ and integrating over $\Omega$) we get
\begin{equation} \label{as.2}
-\int_\Omega (kL\eta + \zeta \eta^2)\Div \vu \dx = \frac{{\rm d}}{{\rm d}t} \int_\Omega \big(kL \eta \ln \eta + \zeta \eta^2\big)\dx + \varepsilon \int_\Omega \Big(\frac{kL}{\eta} + 2\zeta \Big)|\nabla \eta|^2 \dx.
\end{equation}
Next, we apply the matrix operator trace on \eqref{04a} and use the symmetry of $\tn{T}$. Then we get
\begin{equation} \label{as.3}
\int_\Omega \tn{T}:\nabla \vu \dx = \frac{{\rm d}}{{\rm d}t} \frac 12 \int_\Omega {\rm Tr}\, \tn{T} \dx + \frac{1}{4\lambda} \int_\Omega {\rm Tr}\, \tn{T} \dx -\frac{3k}{4\lambda} \int_\Omega \eta \dx.
\end{equation}
Summing up \eqref{as.1}--\eqref{as.3} we end up with
\begin{equation} \label{as.4}
\begin{aligned}
& \frac{{\rm d}}{{\rm d}t} \int_\Omega \Big(\frac 12 \vr |\vu|^2 + P(\vr) + \frac 12 {\rm Tr}\, \tn{T} + kL\eta \ln \eta + \zeta \eta^2 \Big)\dx + \varepsilon\int_\Omega \Big(\frac{kL}{\eta}+ 2\zeta \Big)|\nabla \eta|^2\dx \\
&+ \int_\Omega \Big(2\mu_0(1+ |\tn{D}^d(\vu)|^2)^{\frac{r-2}{2}} |\tn{D}^d(\vu)|^2 + \Lambda (\Div \vu) + \frac{1}{4\lambda} {\rm Tr}\, \tn{T} \Big)\dx \\
&= \int_\Omega \vr\vc{f}\cdot \vu \dx + \frac{3k}{4\lambda}\int_\Omega \eta \dx.
\end{aligned}
\end{equation}
Assuming the extra stress tensor additionally positive definite we can read from \eqref{as.4} the following estimates
\begin{equation} \label{apriori_estimates1}
\begin{aligned}
\|\Div\vu\|_{L^\infty((0,T)\times \Omega))} &+ \|\nabla\vu\|_{L^r((0,T)\times \Omega))} + \|\eta\|_{L^\infty(0,T; L^2(\Omega))} \\
&+ \|\nabla \eta^{\frac 12}\|_{L^2((0,T)\times \Omega))} + \|\nabla \eta\|_{L^2((0,T)\times \Omega))} \leq C(DATA).
\end{aligned}
\end{equation}
Furthermore, assuming boundedness of the initial condition for the density, we also have
\begin{equation} \label{apriori_estimates2}
\|\vr\|_{L^\infty((0,T)\times \Omega))}  \leq C(DATA).
\end{equation}
Next, since the initial density is strictly positive and the initial polymer pressure is non-negative, we also have $\vr\geq C >0$ and $\eta\geq 0$ a.e. in $(0,T)\times \Omega$.
To conclude, we need some estimates of the extra stress tensor. To this aim, we multiply \eqref{04a} on $\tn{T}$ and integrate over $\Omega$. We get
\begin{equation} \label{energy_est_T}
\begin{aligned}
 \frac{{\rm d}}{{\rm d}t} \frac 12 \int_\Omega |\tn{T}|^2 \dx &+ \varepsilon \int_\Omega  |\nabla \tn{T}|^2 \dx + \frac{1}{2\lambda} \int_\Omega |\tn{T}|^2\dx \\
 & = -\frac 12 \int_\Omega \Div \vu \,|\tn{T}|^2 \dx + 2\int_\Omega (\nabla \vu \tn{T}):\tn{T} \dx + \frac{k}{2\lambda} \int_\Omega \eta {\rm Tr}\, \tn{T} \dx.  
\end{aligned}
\end{equation}
While the first term on the right-hand side can be controlled easily, for the second one we have
$$
\Big |\int_\Omega (\nabla \vu \tn{T}):\tn{T} \dx\Big| \leq C \|\nabla \vu\|_{\frac 52} \|\tn{T}\|_{\frac {10}{3}}^2 \leq \min\Big\{\frac \varepsilon 2,\frac{1}{4\lambda}\Big\} \|\tn{T}\|_{1,2}^2 + C \|\nabla \vu\|_{\frac 52}^{\frac 52} \|\tn{T}\|_2^2.
$$
The third term can be easily estimated and we get by Gronwall's argument additional estimate
\begin{equation} \label{apriori_estimates3}
\|\tn{T}\|_{L^\infty(0,T; L^2(\Omega))} + \|\tn{T}\|_{L^2(0,T; W^{1,2}(\Omega))} \leq C(DATA).
\end{equation}
Even though this information is not yet sufficient for the weak compactness of solutions, we stop dealing with formal estimates at this moment and return back to this problem during the construction of solutions.  On the other hand, from the computations above it is evident that at least by the given technique, the minimal value for closing the estimates is the value $r=\frac 52$ for which can get \eqref{apriori_estimates1}--\eqref{apriori_estimates3}. We now start with construction of approximate solutions.

\section{Construction of a weak solution}
\label{s_4}

We first mollify the initial data. Let $S_\Theta$ be the spatial mollifier (all functions below are extended by zero to the whole $\R^3$). We define
$$
\vr_{0,\Theta} := S_\Theta [\vr_0 ], \quad \vu_{0,\Theta}:= S_\Theta[\vu_0], \quad \eta_{0,\Theta} := \Theta + S_\Theta(\eta_0), \quad \tn{T}_{0,\Theta} :=\Theta \tn{I} + S_\Theta[\tn{T}_0].
$$
Then all functions are smooth, moreover, $\vr_{0,\Theta}$ and $\eta_{0,\Theta}$ are additionally strictly positive (the initial density independently of $\Theta$ while for $\eta_{0,\Theta}$ we lose this property for $\Theta \to 0^+$) and $\tn{T}_{0,\Theta}$ is strictly positive definite and symmetric. We also have $\|\vr_{0,\Theta}\|_\infty \leq \|\vr_0\|_\infty$ and for $\Theta \to 0^+$ we further have $\vr_{0,\Theta} \to \vr_0$ in any $L^p(\Omega)$, $1\leq p <\infty$, $\vu_{0,\Theta} \to \vu_0$ in $L^2(\Omega;\R^3)$, $\eta_{0,\Theta} \to \eta_0$ in $L^2(\Omega)$ and $\tn{T}_{0,\Theta} \to \tn{T}_0$ in $L^2(\Omega;\R^{3\times 3})$.

We now take $\alpha >0$ and instead of \eqref{01a}--\eqref{04a} we consider, indeed, in the weak sense, the following problem.
\begin{equation}
\label{01a-alpha}
 \pder{\vr}{t} + \Div (\vr \vu) = 0,
\end{equation}
\vspace{-13pt}
\begin{equation}
\begin{aligned}
\label{02a-alpha}
\pder{ (\vr\vu)}{t} + \Div (\vr \vu \otimes \vu) +\nabla p(\vr)  - \Div \SSS(\nabla_x \vu) &=\Div \big(\TT - (kL\eta  + \de\, \eta^2)\,\II\, \big)  \\
& +  \vr\, \ff -\frac{\alpha}{2} \nabla {\rm Tr}\, (\log \tn{T}),
\end{aligned}
\end{equation}
\vspace{-13pt}
\begin{equation}
\label{03a-alpha}
\pder{\eta}{t} + \Div (\eta \vu) = \e \Delta_x \eta,
\end{equation}
\vspace{-13pt}
\begin{equation}
\label{04a-alpha}
\pder{ \TT}{t} + {\rm Div} (\vu\,\TT) - \left(\nabla \vu \,\TT + \TT\, \nabla^{\rm T} \vu \right) = \e \Delta \TT + \frac{k}{2\lambda}(\eta +\alpha)  \,\II - \frac{1}{2\lambda} \TT,
\end{equation}
with the same boundary condition \eqref{boundary} and the initial conditions $(\vr_{0,\Theta},\vu_{0,\Theta}, \eta_{0,\Theta}, \tn{T}_{0,\alpha,\Theta})$, where $\tn{T}_{0,\alpha,\Theta} = \tn{T}_{0,\Theta}+ \alpha \tn{I}$.

Next we consider three further levels of approximation. First, we consider Galerkin approximation for the velocity. We take $\{\vc{w}_i\}_{i=1}^\infty$ an orthonormal basis in $L^2(\Omega;\R^3)$ which is also a basis in $W^{1,2}_0(\Omega;\R^3)$ formed by eigenfunctions of the Lam\'e system
$$
-\Delta \vc{w}_i -\nabla \Div \vc{w}_i = \lambda_i \vc{w}_i
$$
with homogeneous Dirichlet boundary conditions (for the Navier boundary conditions we can use the conditions $\vc{w}_i\cdot \vc{n} =0$ and $(\nabla \vc{w}_i + \Div \vc{w}_i \tn{I})\vc{n} \cdot \vcg{\tau} = 0$, $\vcg{\tau}$ being the tangent vector to $\partial \Omega$). We denote $$
X_n = {\rm Lin}\, \{\vc{w}_1,\vc{w}_2\dots,\vc{w}_n\}
$$
and look for $\vc{u}_n \in AC([0,T];X_n)$ with $\vc{u}_n(0,\cdot) = P_n(\vc{u}_{0,\Theta})$, $P_n$ being the orthogonal projection from $L^2(\Omega;\R^3)$ to $X_n$, such that 
\begin{equation}
\label{01a-n}
\int_0^t \int_{\R^3} \Big(\vr_n \pder{\psi}{t} + \vr_n\vu_n\cdot \nabla \psi\Big) \dx\dtau = \int_{\R^3} (\vr_n \psi)(t,x)\dx - \int_{\R^3}\vr_{0,\Theta}(x) \psi (0,x) \dx, 
\end{equation}
for all $\psi \in C^\infty_0([0,T]\times \R^3)$, where $\vr_n$ and $\vu_n$ are extended by zero outside of $\Omega$,
\begin{equation}
\begin{aligned}
\label{02a-n}
&\int_0^t \int_\Omega \big(\Lambda (\Div \vcg{\varphi})-\Lambda(\Div \vu_n)\big) \dx \dtau \geq \frac 12 \int_\Omega \big((\vr_n|\vu_n|^2) (t,x) -(\vr_{0,\Theta} |P_n(\vu_{0,\Theta})|^2)(x) \big)\dx \\
&-\int_\Omega (\vr_n\vu_n\cdot \vcg{\varphi})(t,x) \dx   + \int_\Omega (\vr_{0,\Theta} P_n(\vu_{0,\Theta}))(x) \cdot\vcg{\varphi}(0,x)\dx \\
&+ \int_0^t \int_\Omega \Big(\vr_n \vu_n \cdot \pder{\vcg{\varphi}}{t} + \vr_n(\vu_n\otimes \vu_n):\nabla \vcg{\varphi} \Big)\dx\dtau \\
&+ \int_0^t \int_{\Omega} \big(2\mu_0 (1+ |\tn{D}^d(\vu_n)|^2)^{\frac{r-2}{r}}\tn{D}^d(\vu_n) : \tn{D}^d(\vu_n-\vcg{\varphi}) + p(\vr_n) \Div (\vcg{\varphi} -\vu_n)\big) \dx\dtau \\
&+ \int_\Omega \Big( \frac{\alpha}{2} {\rm Tr}\, (\log \tn{T}_n)+ kL\eta_n + \zeta\, \eta_n^2\Big) \Div (\vcg{\varphi}-\vu_n)\dx  
-  \int_\Omega \Big( \tn{T}_n :\nabla (\vcg{\varphi}-\vu_n) - \vr_n\, \ff\cdot (\vcg{\varphi}-\vu_n)\Big)\dx,
\end{aligned}
\end{equation}
\vspace{-13pt}
\begin{equation}
\label{03a-n}
\pder{\eta_n}{t} + \Div (\eta_n \vu_n) = \e \Delta \eta_n,
\end{equation}
where $\pder{\eta_n}{\vc{n}} = 0$ at the boundary of $\Omega$ and $\eta_n(0,x) = \eta_{0,\Theta}(x)$,
\begin{equation}
\label{04a-n}
\pder{ \TT_n}{t} + {\rm Div} (\vu_n \TT_n) - \left(\nabla \vu_n \TT_n + \TT_n \nabla^{\rm T} \vu_n \right) + \frac{1}{2\lambda} \TT_n = \e \Delta \TT_n + \frac{k}{2\lambda}(\eta_n +\alpha)  \II,
\end{equation} 
where $\pder{\TT_n}{\vc{n}} = \tn{O}$ at the boundary of $\Omega$ and $\TT_{n}(0,x) = \TT_{0,\alpha,\Theta}(x)$. Above, equation \eqref{02a-n} is considered for all $\vcg{\varphi} \in C^1([0,T];X_n)$. 
The next step consists in regularization and symmetrization of the extra stress tensor. We take $\sigma >0$ (sufficiently small, in particular, $\sigma < \max\{\alpha, \Theta\}$) and define
$$
\chi_\sigma (s):= \max\{\sigma, s\}, \quad s\in \R.
$$
We now replace most of the terms containing $\tn{T}$ by $\chi_\sigma(\tn{T}^S)$, where we use the matrix calculus for symmetric matrices from Subsection \ref{s_3.1}. The symmetrization of the stress tensor $\tn{T}$ is defined as follows
$$
\tn{T}^S : = \frac 12 (\tn{T} + \tn{T}^T). 
$$
We therefore consider 
\begin{equation}
\label{01a-n-sigma}
\int_0^t \int_{\R^3} \Big(\vr_{n,\sigma} \pder{\psi}{t} + \vr_{n,\sigma}\vu_{n,\sigma}\cdot \nabla \psi\Big) \dx\dtau = \int_{\R^3} (\vr_{n,\sigma} \psi)(t,x)\dx - \int_{\R^3}\vr_{0,\Theta}(x) \psi (0,x) \dx, 
\end{equation}
for all for all $\psi \in C^\infty_0([0,T]\times \R^3)$, where the density and velocity are extended by zero outside of $\Omega$,
\begin{equation}
\begin{aligned}
\label{02a-n-sigma}
&\int_0^t \int_\Omega \big(\Lambda (\Div \vcg{\varphi})-\Lambda(\Div \vu_{n,\sigma})\big) \dx \dtau \geq \frac 12 \int_\Omega \big((\vr_{n,\sigma}|\vu_{n,\sigma}|^2) (t,x) -(\vr_{0,\Theta} |P_n(\vu_{0,\Theta})|^2)(x) \big)\dx \\
&-\int_\Omega (\vr_{n,\sigma}\vu_{n,\sigma}\cdot \vcg{\varphi})(t,x) \dx   + \int_\Omega (\vr_{0,\Theta} P_n(\vu_{0,\Theta}))(x) \cdot\vcg{\varphi}(0,x)\dx \\
&+ \int_0^t \int_\Omega \Big(\vr_{n,\sigma} \vu_{n,\sigma} \cdot \pder{\vcg{\varphi}}{t} + \vr_{n,\sigma}(\vu_{n,\sigma}\otimes \vu_{n,\sigma}):\nabla \vcg{\varphi} \Big)\dx\dtau \\
&+ \int_0^t \int_{\Omega} \big(2\mu_0 (1+ |\tn{D}^d(\vu_{n,\sigma})|^2)^{\frac{r-2}{2}}\tn{D}^d(\vu_{n,\sigma}) : \tn{D}^d(\vu_{n,\sigma}-\vcg{\varphi}) + p(\vr_{n,\sigma}) \Div (\vcg{\varphi} -\vu_{n,\sigma})\big) \dx\dtau \\
& =\int_\Omega \Big( \frac{\alpha}{2} {\rm Tr}\, (\log \chi_\sigma(\tn{T}^S_{n,\sigma}))+ kL\eta_{n,\sigma} + \zeta\, \eta_{n,\sigma}^2\Big) \Div (\vcg{\varphi}-\vu_{n,\sigma})\dx \\  
&+  \int_\Omega \Big( -\chi_\sigma(\tn{T}^S_{n,\sigma}) :\nabla (\vcg{\varphi}-\vu_{n,\sigma}) + \vr_{n,\sigma}\, \ff\cdot (\vcg{\varphi}-\vu_{n,\sigma})\Big)\dx,
\end{aligned}
\end{equation}
\vspace{-13pt}
\begin{equation}
\label{03a-n-sigma}
\pder{ \eta_{n,\sigma}}{t} + \Div (\eta_{n,\sigma} \vu_{n,\sigma}) = \e \Delta \eta_{n,\sigma},
\end{equation}
\vspace{-13pt}
\begin{equation}
\begin{aligned}
\label{04a-n-sigma}
&\pder{ \TT_{n,\sigma}}{t} + {\rm Div} (\vu_{n,\sigma}\,\chi_\sigma(\tn{T}^S_{n,\sigma})) - \left(\nabla \vu_{n,\sigma} \,\chi_\sigma(\tn{T}^S_{n,\sigma}) + \chi_\sigma(\tn{T}^S_{n,\sigma})\, \nabla^{\rm T} \vu_{n,\sigma} \right) \\
&= \e \Delta \TT_{n,\sigma} + \frac{k}{2\lambda}(\eta_{n,\sigma} +\alpha)  \,\II - \frac{1}{2\lambda} \chi_\sigma(\tn{T}^S_{n,\sigma}),
\end{aligned}
\end{equation}
where again, equation \eqref{02a-n-sigma} is considered for all $\vcg{\varphi} \in C^1([0,T];X_n)$, the quantities $\TT_{n,\sigma}$ and $\eta_{n,\sigma}$ satisfy the homogeneous Neumann boundary conditions and the same initial condition as in the previous step. 
The last approximation level includes regularization of the potential $\Lambda$. We take $\delta >0$ sufficiently small ($\delta <\frac 1b$) and consider
$$
\Lambda_\delta (z) := \left\{
\begin{array}{rl} \Lambda (z) \quad & |z| \leq \frac{1}{b} -\delta \\[3pt]
\Lambda(\frac 1b-\delta) + \Lambda'(\frac 1b-\delta)(z-\frac 1b+\delta) \quad & z> \frac 1b-\delta \\[3pt]
\Lambda(-\frac 1b+\delta) + \Lambda'(-\frac 1b+\delta)(z+\frac 1b-\delta) \quad & z< -\frac 1b+\delta.
\end{array}
\right.
$$
Our problem we consider on this last level of approximation reads as follows
\begin{equation}
\label{01a-n-sigma-delta}
\int_0^t \int_{\R^3} \Big(\vr_{n,\sigma,\delta} \pder{\psi}{t} + \vr_{n,\sigma,\delta}\vu_{n,\sigma,\delta}\cdot \nabla \psi\Big) \dx\dtau = \int_{\R^3} (\vr_{n,\sigma,\delta} \psi)(t,x)\dx - \int_{\R^3}\vr_{0,\Theta}(x) \psi (0,x) \dx
\end{equation}
for all for all $\psi \in C^\infty_0([0,T]\times \R^3)$, with the same zero extension of both the density  and velocity,
\begin{equation}
\begin{aligned}
\label{02a-n-sigma-delta}
&\int_\Omega \Big(\pder{(\vr_{n,\sigma,\delta}\vu_{n,\sigma,\delta})}{t}\cdot \vcg{\varphi} - (\vr_{n,\sigma,\delta} \vu_{n,\sigma,\delta} \otimes \vu_{n,\sigma,\delta}):\nabla \vcg{\varphi} - p(\vr_{n,\sigma,\delta}) \Div \vcg{\varphi}\Big)\dx   \\ 
&+ \int_\Omega \Big( 2\mu_0 (1+ |\tn{D}^d(\vu_{n,\sigma,\delta})|^2)^{\frac {r-2}{2}} \tn{D}^d(\vu_{n,\sigma,\delta}): \nabla \vcg{\varphi} + \Lambda_\delta'(\Div \vu_{n,\sigma,\delta}) \Div \vcg{\varphi} \Big) \dx \\
&= \int_\Omega \Big( \frac{\alpha}{2}  {\rm Tr}\, (\log \chi_\sigma(\tn{T}^S_{n,\sigma,\delta}) + kL\eta_{n,\sigma,\delta} + \zeta\, \eta_{n,\sigma,\delta}^2\Big) \Div \vcg{\varphi} \dx  
+  \int_\Omega \big(\vr_{n,\sigma,\delta}\, \ff\cdot \vcg{\varphi}-\chi_\sigma(\tn{T}^S_{n,\sigma,\delta}) :\nabla \vcg{\varphi} \big)\dx,
\end{aligned}
\end{equation}
\vspace{-13pt}
\begin{equation}
\label{03a-n-sigma-delta}
\pder{ \eta_{n,\sigma,\delta}}{t} + \Div (\eta_{n,\sigma,\delta} \vu_{n,\sigma,\delta}) = \e \Delta \eta_{n,\sigma,\delta},
\end{equation}
\vspace{-13pt}
\begin{equation}
\begin{aligned}
\label{04a-n-sigma-delta}
&\pder{\TT_{n,\sigma,\delta}}{t} + {\rm Div} (\vu_{n,\sigma,\delta}\,\chi_\sigma(\tn{T}^S_{n,\sigma,\delta})) - \left(\nabla \vu_{n,\sigma,\delta} \,\chi_\sigma(\tn{T}^S_{n,\sigma,\delta}) + \chi_\sigma(\tn{T}^S_{n,\sigma,\delta})\, \nabla^{\rm T} \vu_{n,\sigma,\delta} \right) \\
&= \e \Delta \TT_{n,\sigma,\delta} + \frac{k}{2\lambda}(\eta_{n,\sigma,\delta} +\alpha)  \,\II - \frac{1}{2\lambda} \chi_\sigma(\tn{T}^S_{n,\sigma,\delta}).
\end{aligned}
\end{equation} 
The equations are considered with the corresponding initial and boundary condition as above. Equality \eqref{02a-n-sigma-delta} holds now a.e. in $(0,T)$ for all $\vcg{\varphi} \in X_n$.
In what follows, we first construct a global in time solution to \eqref{01a-n-sigma-delta}--\eqref{04a-n-sigma-delta} and then subsequently pass with $\delta \to 0^+$, $\sigma \to 0^+$, $n\to \infty$, $\Theta \to 0^+$ and finally $\alpha \to 0^+$.

\subsection{Construction of approximate solution}
\label{s_4.1}

We now fix $\sigma$, $\delta$, $\alpha$ and $\Theta$ positive sufficiently small (recall that $\sigma < \max\{\alpha, \Theta\}$), $n \in \tn{N}$ and construct solutions to \eqref{01a-n-sigma-delta}--\eqref{04a-n-sigma-delta}. Our construction will be performed with the help of the following fixed-point argument of Schauder's type (the proof can be found, e.g., in \cite[Section 9.2 Theorem 4]{EVANS})

\begin{theorem} \label{Schauder}
Let $T$: $X\to X$ be a continuous and compact operator, $X$ is a Banach space. Let for all $s \in [0,1]$ the fixed points $s T(u) =u$ be bounded. Then the operator $T$ possesses at least one fixed point in $X$.
\end{theorem}

We denote $X_n := {\rm Lin}\{\vc{w}_1,\dots,\vc{w}_n\}$ and consider the space $X:= C([0,T];X_n)$. Note that $X_n$ is a finite dimensional space and thus all norms in this space are equivalent. Note further that $X_n \subset W^{2,q}(\Omega) \hookrightarrow C^{1,\beta}(\overline{\Omega})$ (our $q$ can be arbitrary finite; thus we can work with any $0\leq \beta <1$). 

We define the operator $T$ as follows. We take $\vc{w} \in C([0,T];X_n)$ and define $\vc{u} := T(\vc{w})$, where $\vc{w} \mapsto (\vr,\eta) \mapsto \tn{T} \mapsto \vc{u}$. The functions $(\vr,\eta)$ are solutions to
\begin{equation} \label{FP_1}
\int_0^t \int_\Omega \Big(\varrho \pder{\psi}{t} + \vr\vc{w}\cdot \nabla \psi \Big) \dx \,{\rm d}\tau = \int_\Omega \vr(t,\cdot)\psi (t,\cdot)\dx -\int_\Omega \vr_{0,\Theta} \psi(0,\cdot) \dx, 
\end{equation}
\vspace{-8pt}
\begin{equation} \label{FP_2}
\pder{\eta}{t} + \Div(\eta\vc{w}) = \varepsilon \Delta \eta,
\end{equation}
$\eta(0,x) =\eta_{0,\Theta}(x)$, $\pder{\eta}{\vc{n}} =0$ at the boundary of  $\Omega$. Further
\begin{equation} \label{FP_3}
\begin{aligned}
\pder{\tn{T}}{t} + {\rm Div} (\vc{w} \chi_\sigma(\tn{T}^S)) - (\nabla \vc{w} \chi_\sigma(\tn{T}^S) + \chi_\sigma(\tn{T}^S)\nabla^{T}\vc{w}) = \varepsilon \Delta \tn{T} + \frac{k}{2\lambda}(\eta +\alpha)\tn{I} -\frac{1}{2\lambda} \chi_\sigma(\tn{T}^S),
\end{aligned}
\end{equation}
where $\tn{T}(0,x)= \tn{T}_{0,\alpha,\Theta}(x)$, $\pder{\tn{T}}{\vc{n}} = \tn{O}$ at the boundary of $\Omega$, and finally we look for $\vc{u} = \sum_{i=1}^n c_i^n(t) \vc{w}_i(x)$, $\vc{u}(0,x) = P_n(\vc{u}_{0,\Theta})(x)$, where we have for any $\vcg{\varphi} \in X_n$
\begin{equation} \label{FP_4}
\begin{aligned}
&\int_\Omega\Big(\pder{(\vr\vu)}{t} -\vr (\vu\otimes\vc{w}):\nabla \vcg{\varphi} -p(\vr) \Div\vcg{\varphi}\Big) \dx \\
&+ \int_\Omega \Big(2\mu_0(1+ |\tn{D}^d(\vc{w})|^{2})^{\frac{r-2}{2}}{\tn D}^d(\vu) :{\tn D}^d (\vcg{\varphi})  + \Lambda_{\delta}'(\Div\vc{w}) \Div \vcg{\varphi}\Big)\dx \\
&= \int_\Omega \Big(\frac {\alpha}{2} {\rm Tr}\, (\log\chi_\sigma (\tn{T}^S)) + kL\eta +\zeta \eta^2  \Big) \Div \vcg{\varphi} \dx + \int_\Omega (\vr \vc{f}\cdot \vcg{\varphi} -\chi_\sigma(\tn{T^S}):\nabla \vcg{\varphi})\dx.
\end{aligned}
\end{equation}
It is easy to see that for $\vc{w}\in C([0,T];X_n)$ the unique solution to \eqref{FP_1} satisfies
$$
\vr \in L^\infty((0,T)\times \Omega), \quad \vr \geq C_0(T) >0, \quad \text{ and } |\nabla \vr|, \pder{\vr}{t} \in L^\infty(0,T; L^2(\Omega)).
$$
Further, the unique solution to \eqref{FP_2}
$$
\eta \in L^\infty((0,T)\times \Omega), \quad \eta \geq C_0(T) >0, \quad \text{ and } |\nabla^2 \eta|, \pder{\eta}{t} \in L^q((0,T)\times \Omega) \text{ for all } q\in [1,\infty).
$$ 
Next, the unique solution to \eqref{FP_3} $\tn{T}$ is a symmetric tensor and
$$
\tn{T} \in L^\infty((0,T)\times \Omega;\R^{3\times 3}), \quad \text{ and } |\nabla^2 \tn{T}|, \Big|\pder{\tn{T}}{t}\Big| \in L^q((0,T)\times \Omega) \text{ for all } q\in [1,\infty).
$$
Finally, for given $\vc{w}$ and $\vr(\vc{w}), \eta(\vc{w}), \tn{T}(\eta(\vc{w}),\vc{w})$, solutions to \eqref{FP_1}--\eqref{FP_3}, there exists unique $\vc{u}$ solution to \eqref{FP_4}. Since $\vr$ is bounded away from zero and the time derivative is square integrable, the operator $T$ is clearly compact as a mapping from $C([0,T];X_n)$ into itself. It is also not difficult to verify that the mapping is further continuous. To be able to apply Theorem \ref{Schauder}, it remains to verify that the fixed points $sT(\vu)=\vu$ are bounded for all $s\in [0,1]$. 

To this aim, we replace in \eqref{FP_4} $\vu \mapsto \frac{\vu}{s}$ as well as $\vc{w} \mapsto \vu$ and  use as a test function in  this equation the  function $\vu$. We combine the resulted equality with (renormalized versions of) equations \eqref{FP_1} and \eqref{FP_2} and obtain (recall that $P(\vr) = \vr \psi(\vr)$ is the pressure potential)
\begin{equation} \label{E1}
\begin{aligned}
&\frac{{\rm d}}{{\rm d}t} \int_\Omega \Big(\frac 12 \vr|\vu|^2 + s P(\vr) + s(kL \eta \log \eta  + \zeta \eta^2) \Big)\dx \\
& + s \int_\Omega \Big(\varepsilon\Big( \frac{kL}{\eta} + 2 \zeta\Big)|\nabla \eta|^2 + \Lambda_\delta'(\Div \vc{u})\Div \vu\Big)\dx + \int_\Omega 2\mu_0(1+ |\tn{D}^d(\vu)|^2)^{\frac {r-2}{2}}|\tn{D}^d(\vu)|^2 \dx \\
& = s \int_\Omega \Big( \frac \alpha 2 {\rm Tr}\,\log (\chi_\sigma(\tn{T}^S)) \Div \vu + \vr\vc{f}\cdot \vu -\chi_\sigma(\tn{T}^S):\nabla \vu \Big)\dx.
\end{aligned}
\end{equation}
Similarly as in the part devoted to formal a priori estimates, we apply the matrix trace operator on \eqref{FP_3} and integrate over $\Omega$. It yields (note that the solution $\tn{T}$ is symmetric and remains so throughout all limit passages considered below)
\begin{equation} \label{E2}
\frac{{\rm d}}{{\rm d}t} \int_\Omega {\rm Tr}\, \tn{T} \dx + \frac{1}{2\lambda}\int_\Omega {\rm Tr}\, (\chi_\sigma(\tn{T})) \dx = \int_\Omega \Big(2\chi_\sigma (\tn{T}) :\nabla \vu  + \frac {3k}{2\lambda} (\eta +\alpha)\Big)\dx.
\end{equation}
We need one extra estimate since we still do not know that the tensor $\tn{T}$ is positive definite. Following \cite{Barrett-Lu-Suli} we obtain a logarithmic estimate for the trace of the tensor $\tn{T}$.

We define
$$
G_\sigma(s) := \left\{
\begin{array}{rl}
\log s, \quad & \text{ if } s\geq \sigma \\
\frac{s}{\sigma}+ \log \sigma -1, \quad & \text{ if } s< \sigma.
\end{array}
\right.
$$
We easily verify that $G_\sigma'(s) = [\chi_\sigma(s)]^{-1}$ (inverse function). Thus also $G'_\sigma(\tn{P})= [\chi_\sigma(\tn{P})]^{-1}$ for any real symmetric matrix $\tn{P}$. We recall two results; their proofs can be found in \cite{Barrett-Lu-Suli}:
\begin{itemize}
\item Let $g\in C^{1,\gamma}(\tn{R})$, $0<\gamma \leq 1$ be a convex (or concave) function and $\tn{P} \in W^{1,2}(0,T;\tn{R}^{d\times d})$ be symmetric. Then the mapping $t\in (0,T) \mapsto g(\tn{P}(t)) \in \tn{R}^{d \times d}$ is differentiable a.e. and it holds
$$
\frac{{\rm d}{\rm Tr} (g(\tn{P}))}{{\rm d}t} = {\rm Tr}\, \Big(g'(\tn{P}) \frac{{\rm d}\tn{P}}{{\rm d}t}\Big) = \frac{{\rm d}\tn{P}}{{\rm d}t} : g'(\tn{P}).
$$
\item Let $\sigma >0$ and $\tn{P} \in C([0,T]; W^{1,2}(\Omega;\R^{d\times d}))$ be a symmetric matrix such that $\Delta \tn{P} \in L^2(0,T;L^{2}(\Omega;\R^{d\times d}))$ satisfying $\pder{\tn{P}}{\vc{n}} =0$ at $\partial \Omega$. Then $[\chi_\sigma(\tn{P})]^{-1} \in L^\infty(0,T;W^{1,2}(\Omega;\tn{R}^{d\times d}))$ and it holds
$$
\int_\Omega \Delta \tn{P} :[\chi_{\sigma}(\tn{P})]^{-1} \dx = -\int_\Omega \nabla \tn{P}:\nabla [\chi_{\sigma}(\tn{P})]^{-1}\dx \geq \frac 1d \int_\Omega |\nabla {\rm Tr}\,(\log(\chi_\sigma(\tn{P})))|^2\dx
$$
a.e. in $(0,T)$.
\end{itemize}

Based on these results, we have in our case
$$
\begin{aligned}
\pder{\tn{T}}{t}: G_\sigma '(\tn{T}) &= \pder{}{t} \big({\rm Tr}(G_\sigma(\tn{T}))\big) \\
\int_\Omega \Delta \tn{T} : G'_\sigma(\tn{T}) \dx &= -\int_\Omega \nabla \tn{T}:: \nabla [\chi_\sigma(\tn{T})]^{-1} \dx \geq \frac 13 \int_\Omega |\nabla {\rm Tr}\,(\log(\chi_\sigma(\tn{T})))|^2\dx.
\end{aligned}
$$

Whence, it holds
\begin{equation} \label{E3}
\begin{aligned}
\frac{{\rm d}}{{\rm d}t} \int_\Omega {\rm Tr }\, (G_\sigma (\tn{T})) \dx &\geq \int_\Omega {\rm Tr}\, (\log \chi_{\sigma} (\tn{T})) \Div \vu \dx + \frac{\varepsilon}{3} \int_\Omega |\nabla {\rm Tr}\, (\log \chi_{\sigma}(\tn{T}))|^2 \dx \\
& +\frac{k}{2\lambda} \int_\Omega(\eta+\alpha) {\rm Tr}\, [\chi_\sigma(\tn{T})]^{-1} \dx - \frac {3}{2\lambda} |\Omega|.
\end{aligned}
\end{equation}
Above, we used 
$$
{\rm Div}\, (\vu \chi_\sigma(\tn{T})): G'_\sigma(\tn{T})= \vu \cdot \nabla {\rm Tr}\, (\log(\chi_\sigma(\tn{T}))):G'_\sigma(\tn{T}) + 3 \Div \vu
$$
and
$$
-(\nabla \vu \chi_\sigma(\tn{T}) + \chi_\sigma(\tn{T}) \nabla^T \vu)= -2\Div \vu,
$$
see \cite[Section 7]{Barrett-Lu-Suli}. 
We now multiply \eqref{E3} on $-\frac{\alpha s}{2}$ and sum together with \eqref{E1} and \eqref{E2} multiplied by $\frac s2$. We end up with
\begin{equation} \label{E4}
\begin{aligned}
&\frac{{\rm d}}{{\rm d}t} \int_\Omega  \Big( \frac 12 \vr|\vu|^2 + s P(\vr) + s (kL\eta \log \eta +\zeta\eta^2)\Big)\dx + \frac s2  \frac{{\rm d}}{{\rm d}t} \int_\Omega {\rm Tr}\, \big(\tn{T}-\alpha G_\sigma(\tn{T})\big)\dx \\
&+ \varepsilon s \int_\Omega \Big(\frac{kL}{\eta} + 2\zeta \eta\Big) |\nabla \eta|^2 \dx + \frac{\alpha k}{4\lambda}s \int_\Omega (\eta +\alpha){\rm Tr}\, [\chi_\sigma(\tn{T})]^{-1} \dx \\
& + \int_\Omega \big(2\mu_0 (1+|\tn{D}^d(\vu)|^2)^{\frac{r-2}{2}} |\tn{D}^d(\vu)|^2 + s\Lambda_\delta'(\Div \vu) \Div \vu\big)\dx +\frac{s}{4\lambda}  \int_\Omega {\rm Tr}\, (\chi_\sigma(\tn{T}))\dx \\
& + \frac{\alpha \varepsilon s}{6}  \int_\Omega |\nabla {\rm Tr}\, (\log \chi_{\sigma}(\tn{T}))|^2 \dx \leq s \int_\Omega \vr\vc{f}\cdot \vu \dx + \frac{3ks}{4\lambda} \int_\Omega (\eta+\alpha) \dx + \frac{3\alpha s}{4\lambda} |\Omega|.
\end{aligned}
\end{equation}
To conclude, following \cite{Barrett-Lu-Suli} we recall that since $\sigma <\alpha$, we have $s-\alpha G_\sigma(s) \geq \alpha -\alpha \log \alpha$ for all $s\in \R$; it is enough to recall the definition of $G_\sigma(s)$ and consider three cases: $s\leq \sigma <\alpha$, $\sigma <s\leq \alpha$ and $\sigma <\alpha \leq s$. We may also, without loss of generality, assume $\alpha \leq 1$ (later on, we let $\alpha \to 0^+$).
Then
$$
{\rm Tr}\, (\tn{T} -\alpha G_\sigma(\tn{T})) = \sum_{j=1}^3 \big(\lambda^{j} -\alpha G_\sigma(\lambda^{j})\big) \geq 3 (\alpha -\alpha \log \alpha),
$$
where $\lambda^{j}$ are the (real) eigenvalues of the symmetric matrix $\tn{T}$. 

We may therefore conclude our estimates. We get, after integrating over time
\begin{equation} \label{E5}
\begin{aligned}
& \int_\Omega  \Big( \frac 12 \vr|\vu|^2 + s P(\vr) + s (kL\eta \log \eta +\zeta\eta^2)\Big) (t,\cdot)\dx +  \frac 12 s \int_\Omega {\rm Tr}\, (\tn{T} -\alpha G_\sigma(\tn{T}))(t,\cdot) \dx  \\
&+ \varepsilon s \int_0^t \int_\Omega \Big(\frac{kL}{\eta} + 2\zeta\Big) |\nabla \eta|^2 \dx \,{\rm d}\tau +\frac{\alpha k}{4\lambda}s \int_0^t\int_\Omega (\eta +\alpha){\rm Tr}\, [\chi_\sigma(\tn{T})]^{-1} \dx \,{\rm d}\tau  \\
& + \int_0^t\int_\Omega \big(2\mu_0 (1+|\tn{D}^d(\vu)|^2)^{\frac{r-2}{2}} |\tn{D}^d(\vu)|^2 + s\Lambda_\delta'(\Div \vu)\Div \vu \big)\dx \,{\rm d}\tau +\frac{1}{4\lambda} s \int_0^t \int_\Omega {\rm Tr}\, (\chi_\sigma(\tn{T}))\dx\,{\rm d}\tau  \\
& + \frac{\alpha \varepsilon}{6} s \int_0^t \int_\Omega |\nabla {\rm Tr}\, (\log \chi_{\sigma}(\tn{T}))|^2 \dx \,{\rm d}\tau \\
& \leq \int_\Omega  \Big( \frac 12 \vr_{0,\Theta}|\vu_{0,\Theta}|^2 + s P(\vr_{0,\Theta}) + s (kL\eta_{0,\Theta} \log \eta_{0,\Theta} +\zeta\eta_{0,\Theta}^2)\Big) \dx  \\
& + \frac s2  \int_\Omega {\rm Tr}\, (\tn{T}_{0,\alpha,\Theta} -\alpha G_\sigma(\tn{T}_{0,\alpha,\Theta})) \dx +s \int_0^t \int_\Omega \vr\vc{f}\cdot \vu \dx\,{\rm d}\tau + \frac{3k}{4\lambda}s \int_0^t \int_\Omega (\eta+\alpha) \dx \,{\rm d}\tau + \frac{3\alpha T}{4\lambda}s |\Omega|.
\end{aligned}
\end{equation}
The above inequality provides us an estimate of $\vu$ in $C([0,T];L^2(\Omega;\R^3))$ which allows us to claim existence of global in time solution to our approximate problem. The last term on the first line can be bounded from below by  $\frac 32 s  (\alpha-\alpha \log \alpha) |\Omega|$ which ensures its positivity since it is enough to consider $0<\alpha \leq 1$.

\subsection{Estimates independent of $\delta$, limit passage $\delta \to 0^+$}
\label{s_4.2}

Recall that our solution we constructed in the previous subsection satisfies (we use in this section lower index $\delta$ since the aim is to get estimates independent of $\delta$ and to perform the limit passage $\delta \to 0^+$)

\begin{equation} \label{delta_1}
\int_0^t \int_{\R^3} \Big(\varrho_\delta \pder{\psi}{t} + \vr_\delta\vc{u}_\delta\cdot \nabla \psi \Big) \dx \,{\rm d}\tau = \int_{\R^3} \vr_\delta(t,\cdot)\psi (t,\cdot)\dx -\int_{\R^3} \vr_{0,\Theta} \psi(0,\cdot) \dx 
\end{equation}
for all $\psi \in C^\infty_0([0,T];\R^3)$,
\begin{equation} \label{delta_2}
\pder{\eta_\delta}{t} + \Div(\eta_\delta\vc{u}_\delta) = \varepsilon \Delta \eta_\delta,
\end{equation}
$\eta_\delta(0,x) =\eta_{0,\Theta}(x)$, $\pder{\eta_\delta}{\vc{n}} = 0$ at the boundary of $\Omega$. Further (recall that the tensor $\tn{T}$ is symmetric)
\begin{equation} \label{delta_3}
\begin{aligned}
\pder{\tn{T}_\delta}{t} + {\rm Div} (\vc{u}_\delta \chi_\sigma(\tn{T}_\delta)) - (\nabla \vc{u}_\delta \chi_\sigma(\tn{T}_\delta) + \chi_\sigma(\tn{T}_\delta)\nabla^{T}\vc{u}_\delta) = \varepsilon \Delta \tn{T}_\delta + \frac{k}{2\lambda}(\eta_\delta +\alpha)\tn{I} -\frac{1}{2\lambda} \chi_\sigma(\tn{T}_\delta),
\end{aligned}
\end{equation}
where $\tn{T}_\delta(0,x)= \tn{T}_{0,\alpha,\Theta}(x)$, $\pder{\tn{T}}{\vc{n}} = \tn{O}$ at the boundary of $\Omega$, and finally we  have $\vc{u}_\delta = \sum_{i=1}^n c_{i,\delta}^n(t) \vc{w}_i(x)$, $\vc{u}_\delta(0,x) = P_n(\vc{u}_{0,\Theta})(x)$, where we have for any $\vcg{\varphi} \in X_n$
\begin{equation} \label{delta_4}
\begin{aligned}
&\int_\Omega\Big(\pder{(\vr_\delta\vu_\delta)}{t}\cdot \vcg{\varphi} -\vr_\delta (\vu_\delta\otimes\vc{u}_\delta):\nabla \vcg{\varphi} -p(\vr_\delta) \Div\vcg{\varphi}\Big) \dx \\
&+ \int_\Omega \Big(2\mu_0(1+ |\tn{D}^d(\vc{u}_\delta)|^{2})^{\frac{r-2}{2}}{\tn D}^d(\vu_\delta) :{\tn D}^d (\vcg{\varphi})  + \Lambda_{\delta}'(\Div\vc{u}_\delta) \Div \vcg{\varphi}\Big)\dx \\
&= \int_\Omega \Big(\frac {\alpha}{2} {\rm Tr}\, (\log\chi_\sigma (\tn{T}_\delta)) + kL\eta_\delta +\zeta \eta_\delta^2  \Big) \Div \vcg{\varphi} \dx + \int_\Omega (\vr_\delta \vc{f}\cdot \vcg{\varphi} -\chi_\sigma(\tn{T}_\delta):\nabla \vcg{\varphi})\dx
\end{aligned}
\end{equation}
together with the energy inequality
\begin{equation} \label{delta_5}
\begin{aligned}
& \int_\Omega  \Big( \frac 12 \vr_\delta|\vu_\delta|^2 +  P(\vr_\delta) +  (kL\eta_\delta \log \eta_\delta +\zeta\eta_\delta^2)\Big) (t,\cdot)\dx + \frac 12  \int_\Omega {\rm Tr}\, (\tn{T}_\delta -\alpha G_\sigma(\tn{T}_\delta))(t,\cdot) \dx  \\
&+ \varepsilon  \int_0^t \int_\Omega \Big(\frac{kL}{\eta_\delta} + 2\zeta \Big) |\nabla \eta_\delta|^2 \dx \,{\rm d}\tau +\frac{\alpha k}{4\lambda} \int_0^t\int_\Omega (\eta_\delta +\alpha){\rm Tr}\, [\chi_\sigma(\tn{T}_\delta)]^{-1} \dx \,{\rm d}\tau  \\
& + \int_0^t\int_\Omega \big(2\mu_0 (1+|\tn{D}^d(\vu_\delta)|^2)^{\frac{r-2}{2}} |\tn{D}^d(\vu_\delta)|^2 + \Lambda_\delta(\Div \vu_\delta) \big)\dx \,{\rm d}\tau +\frac{1}{4\lambda}  \int_0^t \int_\Omega {\rm Tr}\, (\chi_\sigma(\tn{T}_\delta))\dx\,{\rm d}\tau  \\
& + \frac{\alpha \varepsilon}{6}  \int_0^t \int_\Omega |\nabla {\rm Tr}\, (\log \chi_{\sigma}(\tn{T}_\delta))|^2 \dx \,{\rm d}\tau \\
& \leq \int_\Omega  \Big( \frac 12 \vr_{0,\Theta}|P_n(\vu_{0,\Theta})|^2 +  P(\vr_{0,\Theta}) +  (kL\eta_{0,\Theta} \log \eta_{0,\Theta} +\zeta\eta_{0,\Theta}^2)\Big) \dx  \\
&  + \frac 12  \int_\Omega {\rm Tr}\, (\tn{T}_{0,\alpha,\Theta} -\alpha G_\sigma(\tn{T}_{0,\alpha,\Theta})) \dx +\int_0^t \int_\Omega \vr_\delta\vc{f}\cdot \vu_\delta \dx\,{\rm d}\tau + \frac{3k}{4\lambda} \int_0^t \int_\Omega (\eta_\delta+\alpha) \dx \,{\rm d}\tau + \frac{3\alpha T}{4\lambda} |\Omega|.
\end{aligned}
\end{equation}
Note that we used in the energy inequality the convexity of $\Lambda_\delta$; more precisely, since $\Lambda_\delta(0)=0$, we have $\Lambda'_\delta (\Div\vu_\delta) \Div \vu_\delta \geq \Lambda_\delta(\Div \vu_\delta)$. Recall also that the last term on the first line is non-negative. From \eqref{delta_5} we may get the following estimates independent of $\delta$:
\begin{equation} \label{delta_6}
\begin{aligned}
&\|\vr_\delta |\vu_\delta|^2\|_{L^\infty(0,T;L^1(\Omega))} + \|\eta_\delta\|_{L^\infty(0,T;L^2(\Omega))} + \|\nabla \eta_\delta\|_{L^2((0,T)\times \Omega)} + \|\nabla \eta_\delta^{\frac 12}\|_{L^2((0,T)\times \Omega)} \\
&+ \|\vu_\delta\|_{L^r(0,T;W^{1,r}(\Omega))} \leq C. 
\end{aligned}
\end{equation}
We further also have $\eta_\delta \geq C(T) >0$, $\vr_\delta\geq C(T) > 0$, $\|\vr_\delta\|_{L^\infty((0,T)\times \Omega)} \leq C$ and $\|\eta_\delta\|_{L^\infty((0,T)\times \Omega)} \leq C$ which follow from the fact that all spatial norms for the velocity are equivalent, therefore $\|\nabla \vu_\delta\|_{L^r(0,T;L^\infty(\Omega))} \leq C$ and thus we may use standard estimates of the continuity and regularized continuity equations (for $\vr_\delta$ and $\eta_\delta$, respectively). But then also $\vu_\delta$ is bounded in the time spaces cylinder, due to the bound of the kinetic energy. 

Using Lemma \ref{parabolic_Neumann}, we also have
\begin{equation} \label{delta_7}
\begin{aligned}
&\|\tn{T}_\delta\|_{L^\infty(0,T;W^{1,2}(\Omega))} + \|\tn{T}_\delta\|_{L^2(0,T;W^{2,2}(\Omega))} + \Big\|\pder{\tn{T}_\delta}{t}\Big\|_{L^2((0,T)\times \Omega)} \\
&+ \|\eta_\delta\|_{L^\infty(0,T;W^{1,2}(\Omega))} + \|\eta_\delta\|_{L^2(0,T;W^{2,2}(\Omega))} + \Big\|\pder{\eta_\delta}{t}\Big\|_{L^2((0,T)\times \Omega)} \leq C
\end{aligned}
\end{equation}
and therefore, due to the orthogonality of the basis function, also
\begin{equation} \label{delta_8}
\Big\|\pder{\vu_\delta}{t}\Big\|_{L^1(0,T;L^2(\Omega))} \leq C.
\end{equation}
Since by properties of the continuity equation (cf. \cite{DiPerna-Lions}) we also get the strong convergence of the sequence of density in $C([0,T];L^q(\Omega))$ for any $q<\infty$, we may easily perform the limit passages ("removing $\delta$") in \eqref{delta_1}--\eqref{delta_3}. However, before passing to the limit in the momentum equation, we perform small changes in the formulation. We in fact use in \eqref{delta_4} as test function $\vcg{\varphi}-\vu_\delta$ with $\vcg{\varphi} \in C^1([0,T];X_n)$. Combining it with the continuity equation for the density as well as using the convexity of the potential $\Lambda_\delta$ we end up with 
\begin{equation} \label{delta_9}
\begin{aligned}
\int_0^t &\int_\Omega \big(\Lambda_\delta(\Div \vcg{\varphi}) -\Lambda_\delta(\Div\vu_\delta)\big)\dx\dtau \geq \frac 12 \int_\Omega \vr_\delta |\vu_\delta|^2 (t,\cdot) \dx - \int_\Omega \vr_{0,\Theta}|\vu_{0,\Theta}|^2 \dx \\
& - \int_\Omega \vr_\delta \vu_\delta \cdot \vcg\varphi(t,\cdot)\dx + \int_\Omega \vr_{0,\Theta}\vu_{0,\Theta} \cdot \vcg{\varphi}(0,\cdot)\dx + \int_0^t \int_\Omega \Big(\vr_\delta \vu_\delta \cdot \pder{\vcg{\varphi}}{t} + \vr_\delta \vu_\delta \otimes \vu_\delta : \nabla \vcg{\varphi}\Big)\dx\dtau \\
&+ \int_0^t\int_\Omega 2\mu_0(1+ |\tn{D}^d(\vc{u}_\delta)|^{2})^{\frac{r-2}{2}}{\tn D}^d(\vu_\delta) :{\tn D}^d (\vu_\delta-\vcg{\varphi})  \dx\dtau + \int_0^t \int_\Omega p(\vr_\delta) \Div (\vcg{\varphi}-\vu_\delta)\dx\dtau \\
&+ \int_0^t \int_\Omega \Big(\frac {\alpha}{2} {\rm Tr}\, (\log\chi_\sigma (\tn{T}_\delta)) + kL\eta_\delta +\zeta \eta_\delta^2  \Big) \Div (\vcg{\varphi}-\vu_\delta) \dx\dtau  \\
&+\int_0^t\int_\Omega (\vr_\delta \vc{f}\cdot (\vcg{\varphi}-\vu_\delta) -\chi_\sigma(\tn{T}_\delta):\nabla (\vcg{\varphi}-\vu_\delta))\dx\dtau.
\end{aligned}
\end{equation}
Based on the estimates above combined with inequalities
$$
\begin{aligned}
\Lambda(\Div\vcg{\varphi}) &\geq \Lambda_\delta(\Div\vcg{\varphi}) \qquad \text{ provided } \delta \in (0,\tfrac 1b),\\
\liminf_{\delta \to 0^+} \Lambda_\delta (\Div\vu_\delta) &\geq \Lambda(\Div\vu) \text { provided } \vu_\delta \rightharpoonup \vu \text{ in } L^r(0,T; W^{1,r}(\Omega)),
\end{aligned}
$$
(for more details, see \cite{Feireisl-Liu-Malek}), we may pass to the limit (remove the lower index $\delta$ everywhere)  both in \eqref{delta_9} as well as in \eqref{delta_5}. The proof of the limit passage is finished.

\subsection{Estimates independent of $\sigma$, limit passage $\sigma \to 0^+$}
\label{s_4.3}

We now intend to pass to the limit $\sigma \to 0^+$. The limit passage is quite similar to the previous one with one exception. We namely lose the cut-off function which ensured us that we formally work with positive definite function. So we need to verify that the limit stress tensor $\tn{T}$ is positive definite.

Our system of equations reads as follows:
\begin{equation} \label{sigma_1}
\int_0^t \int_{\R^3} \Big(\varrho_\sigma \pder{\psi}{t} + \vr_\sigma\vc{u}_\sigma\cdot \nabla \psi \Big) \dx \,{\rm d}\tau = \int_{\R^3} \vr_\sigma(t,\cdot)\psi (t,\cdot)\dx -\int_{\R^3} \vr_{0,\Theta} \psi(0,\cdot) \dx 
\end{equation}
for all $\psi \in C^\infty_0([0,T]\times \R^3)$ and both the density and velocity are extended by zero outside of $\Omega$,
\begin{equation} \label{sigma_2}
\pder{\eta_\sigma}{t} + \Div(\eta_\sigma\vc{u}_\sigma) = \varepsilon \Delta \eta_\sigma,
\end{equation}
$\eta_\sigma(0,x) =\eta_{0,\Theta}(x)$, $\pder{\eta_\sigma}{\vc{n}} = 0$ at the boundary of $\Omega$. Further
\begin{equation} \label{sigma_3}
\begin{aligned}
\pder{\tn{T}_\sigma}{t} + {\rm Div} (\vc{u}_\sigma \chi_\sigma(\tn{T}_\sigma)) - (\nabla \vc{u}_\sigma \chi_\sigma(\tn{T}_\sigma) + \chi_\sigma(\tn{T}_\sigma)\nabla^{T}\vc{u}_\sigma) = \varepsilon \Delta \tn{T}_\sigma + \frac{k}{2\lambda}(\eta_\sigma +\alpha)\tn{I} -\frac{1}{2\lambda} \chi_\sigma(\tn{T}_\sigma),
\end{aligned}
\end{equation}
where $\tn{T}_\sigma(0,x)= \tn{T}_{0,\alpha,\Theta}(x)$, $\pder{\tn{T}}{\vc{n}} = \tn{O}$ at the boundary of $\Omega$, and finally we  have $\vc{u}_\sigma = \sum_{i=1}^n c_{i,\sigma}^n(t) \vc{w}_i(x)$, $\vc{u}_\sigma(0,x) = P_n(\vc{u}_{0,\Theta})(x)$, where for any $\vcg{\varphi} \in C([0,T];X_n)$
\begin{equation} \label{sigma_4}
\begin{aligned}
\int_0^t &\int_\Omega \big(\Lambda(\Div \vcg{\varphi}) -\Lambda(\Div\vu_\sigma)\big)\dx\dtau \geq \frac 12 \int_\Omega \vr_\sigma |\vu_\sigma|^2 (t,\cdot) \dx - \int_\Omega \vr_{0,\Theta}|\vu_{0,\Theta}|^2 \dx \\
& - \int_\Omega \vr_\sigma \vu_\sigma \cdot \vcg\varphi(t,\cdot)\dx + \int_\Omega \vr_{0,\Theta}\vu_{0,\Theta} \cdot \vcg{\varphi}(0,\cdot)\dx + \int_0^t \int_\Omega \Big(\vr_\sigma \vu_\sigma \cdot \pder{\vcg{\varphi}}{t} + \vr_\sigma \vu_\sigma \otimes \vu_\sigma : \nabla \vcg{\varphi}\Big)\dx\dt \\
&+ \int_0^t\int_\Omega 2\mu_0(1+ |\tn{D}^d(\vc{u}_\sigma)|^{2})^{\frac{r-2}{2}}{\tn D}^d(\vu_\sigma) :{\tn D}^d (\vu_\sigma-\vcg{\varphi})  \dx\dtau + \int_0^t \int_\Omega p(\vr_\sigma) \Div (\vcg{\varphi}-\vu_\sigma)\dx\dtau \\
&+ \int_0^t \int_\Omega \Big(\frac {\alpha}{2} {\rm Tr}\, (\log\chi_\sigma (\tn{T}_\sigma)) + kL\eta_\sigma +\zeta \eta_\sigma^2  \Big) \Div (\vcg{\varphi}-\vu_\sigma) \dx\dtau  \\
&+\int_0^t\int_\Omega (\vr_\sigma \vc{f}\cdot (\vcg{\varphi}-\vu_\sigma) -\chi_\sigma(\tn{T}_\sigma):\nabla (\vcg{\varphi}-\vu_\sigma))\dx\dtau
\end{aligned}
\end{equation}
together with the energy inequality
\begin{equation} \label{sigma_5}
\begin{aligned}
& \int_\Omega  \Big( \frac 12 \vr_\sigma|\vu_\sigma|^2 +  P(\vr_\sigma) +  (kL\eta_\sigma \log \eta +\zeta\eta_\sigma^2)\Big) (t,\cdot)\dx + \frac 12  \int_\Omega {\rm Tr}\, (\tn{T}_\sigma -\alpha G_\sigma(\tn{T}_\sigma))(t,\cdot) \dx  \\
&+ \varepsilon  \int_0^t \int_\Omega \Big(\frac{kL}{\eta_\sigma} + 2\zeta \Big) |\nabla \eta_\sigma|^2 \dx \,{\rm d}\tau +\frac{\alpha k}{4\lambda} \int_0^t\int_\Omega (\eta_\sigma +\alpha){\rm Tr}\, [\chi_\sigma(\tn{T_\sigma})]^{-1} \dx \,{\rm d}\tau  \\
& + \int_0^t\int_\Omega \big(2\mu_0 (1+|\tn{D}^d(\vu_\sigma)|^2)^{\frac{r-2}{2}} |\tn{D}^d(\vu_\sigma)|^2 + \Lambda(\Div \vu_\sigma) \big)\dx \,{\rm d}\tau +\frac{1}{4\lambda}  \int_0^t \int_\Omega {\rm Tr}\, (\chi_\sigma(\tn{T_\sigma}))\dx\,{\rm d}\tau  \\
& + \frac{\alpha \varepsilon}{6}  \int_0^t \int_\Omega |\nabla {\rm Tr}\, (\log \chi_{\sigma}(\tn{T_\sigma}))|^2 \dx \,{\rm d}\tau \\
& \leq \int_\Omega  \Big( \frac 12 \vr_{0,\Theta}|P_n(\vu_{0,\Theta})|^2 +  P(\vr_{0,\Theta}) +  (kL\eta_{0,\Theta} \log \eta_{0,\Theta} +\zeta\eta_{0,\Theta}^2)\Big) \dx  \\
&  + \frac 12  \int_\Omega {\rm Tr}\, (\tn{T}_{0,\alpha,\Theta} -\alpha G_\sigma(\tn{T}_{0,\alpha,\Theta})) \dx +\int_0^t \int_\Omega \vr_\sigma\vc{f}\cdot \vu_\sigma \dx\,{\rm d}\tau + \frac{3k}{4\lambda} \int_0^t \int_\Omega (\eta_\sigma+\alpha) \dx \,{\rm d}\tau + \frac{3\alpha T}{4\lambda} |\Omega|.
\end{aligned}
\end{equation}
Note that the index $n$ is fixed, whence $\vu_\sigma$ belongs to $C([0,T];X_n)$ for any $\sigma >0$. Therefore all the spatial norms of the velocity sequence are equivalent. Due to the properties of the convex potential $\Lambda(\cdot)$ we also have that
$$
\|\Div \vu_\sigma\|_{L^\infty((0,T)\times \Omega)} \leq \frac 1b.
$$
Due to \eqref{sigma_5} and the properties mentioned above we have the following bounds uniform in $\sigma$:
\begin{equation} \label{sigma_6}
\begin{aligned}
&\|\vu_\sigma\|_{L^\infty(0,T;W^{2,q}(\Omega))} + \|\eta_\sigma\|_{L^\infty(0,T;L^2(\Omega))} + \|\vr_\sigma\|_{L^\infty((0,T)\times \Omega)} + \|\nabla \eta_\sigma\|_{L^2((0,T)\times \Omega)} \\
&+ \|\nabla {\rm Tr}\, \chi_\sigma(\tn{T_\sigma})\|_{L^2((0,T)\times \Omega)} + \int_0^T \int_\Omega {\rm Tr}\, [\chi_\sigma(T_\sigma)]^{-1} \dx\dt \leq C;
\end{aligned}
\end{equation}
above, $q$ is arbitrarily large, but finite. From \eqref{sigma_1}--\eqref{sigma_4} we can obtain
\begin{equation} \label{sigma_7}
\begin{aligned}
& \vr_\sigma \geq C > 0 \qquad \text{ a.e. in } (0,t)\times \Omega \\ 
& \eta_\sigma \geq C > 0 \qquad \text{ a.e. in } (0,t)\times \Omega \\ 
&\|\tn{T}_\sigma\|_{L^\infty(0,T;W^{1,2}(\Omega))} + \|\tn{T}_\sigma\|_{L^2(0,T;W^{2,2}(\Omega))} + \Big\|\pder{\tn{T}_\sigma}{t}\Big\|_{L^2((0,T)\times \Omega)} \\
&+\|\eta_\sigma\|_{L^\infty(0,T;W^{1,2}(\Omega))} + \|\eta_\sigma\|_{L^2(0,T;W^{2,2}(\Omega))} + \Big\|\pder{\eta_\sigma}{t}\Big\|_{L^2((0,T)\times \Omega)} \leq C.
\end{aligned}
\end{equation}
By means of the Aubin--(Simon)--Lions argument, we may extract subsequences (we relabel them) such that ($q<\infty$, arbitrarily large)
\begin{equation} \label{sigma_8}
\begin{aligned}
\vu_\sigma&\to \vu &\qquad& \text{ in } C([0,T];W^{1,q}(\Omega;\R^3)) \\
\tn{T}_\sigma &\to \tn{T} &\qquad& \text{ in } L^q(0,T;W^{1,q}(\Omega;\R^{3\times 3})) \cap  C([0,T)\times \overline{\Omega};\R^{3\times 3})\\ 
\eta_\sigma &\to \eta &\qquad& \text{ in } L^q(0,T;W^{1,q}(\Omega)) \cap  C([0,T)\times \overline{\Omega}) \\
\vr_\sigma&\to \vr &\qquad& \text{ in } C([0,T];L^{q}(\Omega))
\end{aligned}
\end{equation}
(the last line follows from properties of the continuity equation as in the previous step, see \cite{DiPerna-Lions}). 

Let us now verify that the limit tensor $\tn{T}$ is not only symmetric, but also positive definite. Assume the contrary, i.e., there exists $D\subset (0,T]\times \Omega$ with $|D|_4 >0$ (the four dimensional Lebesgue measure) such that $[T]^+ \vc{q} = 0$ a.e. in $(0,T]\times \Omega$ with $|\vc{q}|=1$ a.e. in $D$ and $\vc{q}=\vc{0}$ a.e. in $(0,T]\times \Omega \setminus D$ (the vector $\vc{q} \in L^\infty((0,T)\times \Omega;\R^3)$; note that due to the estimate of the integral term it is not possible that all eigenvalues of $\tn{T}$ become simultaneously zero). Then
$$
\begin{aligned}
&|D| = \int_0^T\int_\Omega |\vc{q}|^2 \dx\dt = \int_0^T \int_\Omega [\chi_\sigma(\tn{T}_\sigma)]^{-\frac 12} \vc{q}\cdot [\chi_\sigma(\tn{T}_\sigma)]^{-\frac 12} \vc{q} \dx \dt \\
&\leq \Big(\int_0^T\int_\Omega [\chi_\sigma(\tn{T}_\sigma)]^{-1} \dx\dt\Big)^{\frac 12} \Big(\int_0^T \int_\Omega \vc{q}^T [\chi_\sigma(\tn{T}_\sigma)]\vc{q} \dx\dt\Big)^{\frac 12} \\
&\leq C \Big(\int_0^T \int_\Omega \vc{q}^T [\chi_\sigma(\tn{T}_\sigma)]\vc{q} \dx\dt\Big)^{\frac 12}.
\end{aligned}
$$
Passing to the limit with $\sigma \to 0$ we finally get that $|D|=0$. Then also 
$$
\nabla {\rm Tr}\, (\log \chi_\sigma (\tn{T}_\sigma)) \rightharpoonup \nabla {\rm Tr}\, (\log \tn{T})
$$
in $L^2((0,T)\times \Omega;\R^3)$.

We may now pass to the limit in the relations \eqref{sigma_1}--\eqref{sigma_4}, where we use that
$$
\liminf_{\sigma \to 0^+} \Lambda (\Div\vu_\sigma) \geq \Lambda (\Div\vu) \qquad \text{ a.e. in } (0,T)\times \Omega
$$
as well as the weak lower semicontinuity of several terms in the energy inequality. The limit passage is therefore completed. 

\subsection{Estimates independent of $n$, limit passage $n \to \infty$}
\label{s_4.4}

The aim of this subsection is to pass to the limit with $n\to \infty$. This limit passage is more complex since we lose the control of the compactness of the sequence of velocities and consequently also of the sequence of densities. Therefore new ideas must be used and in this direction we combine the approach from \cite{Barrett-Lu-Suli} with the approach from \cite{Feireisl-Liu-Malek}. First we recall the system of equations/inequalities satisfied for any $n\in {\mathbb N}$: 

\begin{equation} \label{en_1}
\int_0^t \int_{\R^3} \Big(\varrho_n \pder{\psi}{t} + \vr_n\vc{u}_n\cdot \nabla \psi \Big) \dx \,{\rm d}\tau = \int_{\R^3} \vr_n(t,\cdot)\psi (t,\cdot)\dx -\int_{\R^3} \vr_{0,\Theta} \psi(0,\cdot) \dx 
\end{equation}
for any $\psi \in C^\infty_0([0,T] \times \R^3)$ with the density and velocity extended by zero outside of $\Omega$,
\begin{equation} \label{en_2}
\pder{\eta_n}{t} + \Div(\eta_n\vc{u}_n) = \varepsilon \Delta \eta_n,
\end{equation}
$\eta_\sigma(0,x) =\eta_{0,\Theta}(x)$ and $\pder{\eta_\sigma}{\vc{n}} = 0$ at the boundary of $\Omega$. Further
\begin{equation} \label{en_3}
\begin{aligned}
\pder{\tn{T}_n}{t} + {\rm Div} (\vc{u}_n \tn{T}_n) - (\nabla \vc{u}_n \tn{T}_n + \tn{T}_n \nabla^{T}\vc{u}_n) = \varepsilon \Delta \tn{T}_n + \frac{k}{2\lambda}(\eta_n +\alpha)\tn{I} -\frac{1}{2\lambda} \tn{T}_n,
\end{aligned}
\end{equation}
where $\tn{T}_n(0,x)= \tn{T}_{0,\alpha,\Theta}(x)$, $\pder{\tn{T}}{{\vc{n}}}=\tn{O}$ at the boundary of $\Omega$, and finally we  have $\vc{u}_n = \sum_{i=1}^n c_{i,n}^n(t) \vc{w}_i(x)$, $\vc{u}_n(0,x) = P_n(\vc{u}_{0,\Theta})(x)$, where for any $\vcg{\varphi} \in C([0,T];X_n)$
\begin{equation} \label{en_4}
\begin{aligned}
\int_0^t &\int_\Omega \big(\Lambda(\Div \vcg{\varphi}) -\Lambda(\Div\vu_n)\big)\dx\dtau \geq \frac 12 \int_\Omega \vr_n |\vu_n|^2 (t,\cdot) \dx - \int_\Omega \vr_{0,\Theta}|\vu_{0,\Theta}|^2 \dx \\
& - \int_\Omega \vr_n \vu_n \cdot \vcg\varphi(t,\cdot)\dx + \int_\Omega \vr_{0,\Theta}\vu_{0,\Theta} \cdot \vcg{\varphi}(0,\cdot)\dx + \int_0^t \int_\Omega \Big(\vr_n \vu_n \cdot \pder{\vcg{\varphi}}{t} + \vr_n \vu_n \otimes \vu_n : \nabla \vcg{\varphi}\Big)\dx\dt \\
&+ \int_0^t\int_\Omega 2\mu_0(1+ |\tn{D}^d(\vc{u}_n)|^{2})^{\frac{r-2}{2}}{\tn D}^d(\vu_n) :{\tn D}^d (\vu_n-\vcg{\varphi})  \dx\dtau + \int_0^t \int_\Omega p(\vr_n) \Div (\vcg{\varphi}-\vu_n)\dx\dtau \\
&+ \int_0^t \int_\Omega \Big(\frac {\alpha}{2} {\rm Tr}\, (\log \tn{T}_n) + kL\eta_n +\zeta \eta_n^2  \Big) \Div (\vcg{\varphi}-\vu_n) \dx\dtau  \\
&+\int_0^t\int_\Omega (\vr_n \vc{f}\cdot (\vcg{\varphi}-\vu_n) -\tn{T}_n:\nabla (\vcg{\varphi}-\vu_n))\dx\dtau.
\end{aligned}
\end{equation}

This system is accompanied with the energy inequality
\begin{equation} \label{en_5}
\begin{aligned}
& \int_\Omega  \Big( \frac 12 \vr_n|\vu_n|^2 +  P(\vr_n) +  (kL\eta_n \log \eta_n +\zeta\eta_n^2)\Big) (t,\cdot)\dx + \frac 12  \int_\Omega {\rm Tr}\, (\tn{T}_n -\alpha G_\sigma(\tn{T}_n))(t,\cdot) \dx  \\
&+ \varepsilon  \int_0^t \int_\Omega \Big(\frac{kL}{\eta_n} + 2\zeta \Big) |\nabla \eta_n|^2 \dx \,{\rm d}\tau +\frac{\alpha k}{4\lambda} \int_0^t\int_\Omega (\eta_n +\alpha){\rm Tr}\, [\tn{T}_n)]^{-1} \dx \,{\rm d}\tau  \\
& + \int_0^t\int_\Omega \big(2\mu_0 (1+|\tn{D}^d(\vu_n)|^2)^{\frac{r-2}{2}} |\tn{D}^d(\vu_n)|^2 + \Lambda(\Div \vu_n) \big)\dx \,{\rm d}\tau +\frac{1}{4\lambda}  \int_0^t \int_\Omega {\rm Tr}\, \tn{T}_n\dx\,{\rm d}\tau  \\
& + \frac{\alpha \varepsilon}{6}  \int_0^t \int_\Omega |\nabla {\rm Tr}\, \tn{T}_n|^2 \dx \,{\rm d}\tau \\
& \leq \int_\Omega  \Big( \frac 12 \vr_{0,\Theta}|P_n(\vu_{0,\Theta})|^2 +  P(\vr_{0,\Theta}) +  (kL\eta_{0,\Theta} \log \eta_{0,\Theta} +\zeta\eta_{0,\Theta}^2)\Big) \dx  \\
&  + \frac 12  \int_\Omega {\rm Tr}\, (\tn{T}_{0,\alpha,\Theta} -\alpha G_\sigma(\tn{T}_{0,\alpha,\Theta})) \dx +\int_0^t \int_\Omega \vr_n\vc{f}\cdot \vu_n \dx\,{\rm d}\tau + \frac{3k}{4\lambda} \int_0^t \int_\Omega (\eta_n+\alpha) \dx \,{\rm d}\tau + \frac{3\alpha T}{4\lambda} |\Omega|.
\end{aligned}
\end{equation}

We now summarize the estimate which we have for our sequences. They mostly come from the energy inequality, however, below we deduce also several other estimates directly from the equations.
\begin{equation}\label{en_6}
\begin{aligned}
&\|\vr_n |\vu_n|^2\|_{L^\infty(0,T;L^1(\Omega))} + \|\vu_n\|_{L^r(0,T;W^{1,r}_0(\Omega))} + \|\Div \vu_n\|_{L^\infty((0,T)\times \Omega)} + \|\eta_n\|_{L^\infty(0,T;L^2(\Omega))}  \\
&+ \|\eta_n\|_{L^2(0,T;W^{1,2}(\Omega))} + \|\eta_n^{\frac 12}\|_{L^2(0,T;W^{1,2}(\Omega))} + \|\eta_n\|_{L^2(0,T;W^{1,2}(\Omega))} + \|{\rm Tr}\, (\tn{T}_n - \alpha \log \tn{T}_n)\|_{L^\infty(0,T;L^1(\Omega))}  \\
&+\|(\eta_n +\alpha){\rm Tr}\, [\tn{T}_n]^{-1}\|_{L^1((0,T)\times \Omega)} + \|\nabla {\rm Tr}\, \log \tn{T}_n\|_{L^2((0,T)\times \Omega)} \leq C.
\end{aligned}
\end{equation}
Furthermore, the continuity equation and the regularized continuity equation imply
\begin{equation} \label{en_7}
\|\vr_n\|_{L^\infty((0,T)\times \Omega)} \leq C, \qquad \vr_n,\eta_n \geq C > 0 \quad \text{ a.e. in } (0,T)\times \Omega.
\end{equation}
Since the density is uniformly bounded away from zero, we therefore have from the first term on the left-hand side in \eqref{en_6}
$$
\|\vu_n\|_{L^\infty(0,T;L^2(\Omega))} \leq C
$$
which by interpolation implies
\begin{equation} \label{en_9}
\|\vu_n\|_{L^{\frac{2r}{r-1}}((0,T)\times \Omega)} \leq C
\end{equation}
provided $r\geq \frac {11}{5}$ (which is in our case surely satisfied since due to other reasons we need $r\geq \frac 52$). This is precisely the place why it is needed in \cite{Feireisl-Liu-Malek} $r\geq \frac{11}{5}$ as well as the density bounded away from zero, which can be ensured only if $\Div \vu_n$ is bounded.

We now, based on the estimates above, can turn our attention to the bounds of the stress tensor $\tn{T}_n$. Multiplying \eqref{en_3} by $\tn{T}_n$ and using the bound on $\nabla \vu_n$ in $L^r((0,T)\times \Omega)$ together with the Gronwall argument, exactly as in Subsection \ref{s_3.2}, we obtain that
\begin{equation} \label{en_10}
\|\tn{T}_n\|_{L^\infty(0,T;L^2(\Omega))} + \|\nabla \tn{T}_n\|_{L^2((0,T)\times\Omega)} \leq C.
\end{equation}
To conclude, looking once more at the continuity equation, we also get
\begin{equation} \label{en_11}
\Big\|\pder{\vr_n}{t}\Big\|_{L^{\frac{2r}{r-1}}(0,T; W^{-1,\frac{2r}{r-1}} (\Omega))} \leq C
\end{equation}
while from their parabolic regularisations
\begin{equation} \label{en_12}
\Big\|\pder{\eta_n}{t}\Big\|_{L^2(0,T;W^{-1,2}(\Omega))} + \Big\|\pder{\tn{T}_n}{t}\Big\|_{L^{\frac{10}{7}}(0,T;W^{-1,\frac{10}{7}}(\Omega))}\leq C. 
\end{equation}
Then (formally for subsequences, which we can, however, always relabel)
\begin{equation} \label{en_13}
\begin{aligned}
\vr_n \rightharpoonup ^* \vr & \qquad  & & \text{ in } L^\infty((0,T)\times \Omega) \\
\vr_n \to  \vr & \qquad  & & \text{ in } C([0,T];L^q_{{\rm weak}}(\Omega)) \quad \text{ for all } q<\infty \\
\eta_n \rightharpoonup ^* \eta & \qquad & &  \text{ in } L^\infty(0,T;L^2(\Omega)) \cap L^2(0,T;W^{1,2}(\Omega)) \\
\eta_n \to \eta & \qquad & &  \text{ in } L^q((0,T)\times \Omega) \quad \text{ for all } q<\frac{10}{3} \\
\eta_n \to  \eta & \qquad  & & \text{ in } C([0,T];L^2_{{\rm weak}}(\Omega)) \\
\vu_n \rightharpoonup  \vu & \qquad & &  \text{ in } L^{\frac{2r}{r-1}}((0,T)\times\Omega;\R^3) \cap L^r(0,T;W^{1,r}(\Omega;\R^3)) \\
\vu_n \rightharpoonup ^* \vu & \qquad & &  \text{ in } L^\infty(0,T;L^2(\Omega;\R^3))  \\
\Div \vu_n \rightharpoonup ^* \Div \vu & \qquad  & & \text{ in } L^\infty((0,T)\times \Omega) \\
\tn{T}_n \rightharpoonup ^* \tn{T} & \qquad & &  \text{ in } L^\infty(0,T;L^2(\Omega;\R^{3\times 3})) \cap L^2(0,T;W^{1,2}(\Omega;\R^{3\times 3})) \\
\tn{T}_n \to \tn{T} & \qquad & &  \text{ in } L^q((0,T)\times \Omega;\R^{3\times 3}) \quad \text{ for all } q<\frac{10}{3} \\
\tn{T}_n \to  \tn{T} & \qquad  & & \text{ in } C([0,T];L^2_{{\rm weak}}(\Omega;\R^{3\times3})) \\
{\rm Tr}\, \log \tn{T}_n \to \overline{{\rm Tr}\, \log \tn{T}} & \qquad & &  \text{ in } L^2(0,T; W^{1,2}(\Omega)).  
\end{aligned}
\end{equation}
Furthermore, $\eta$, $\vr \geq C >0$ a.e. in $(0,T)\times \Omega$. Next, we also have
$$
\Big\|\pder{}{t}\int_\Omega \vr_n\vu_n \cdot \vc{w}_i \dx\Big\|_{L^1(0,T)} \leq C
$$
for all $i \in \mathbb{N}$ (basis vectors) and thus for any $i \in \mathbb{N}$
$$
\int_\Omega \vr_n \vu_n \cdot \vc{w}_i \dx \to \int_\Omega \vr\vu\cdot \vc{w} \dx \qquad \text {in } L^1(0,T).
$$
Therefore
$$
\vr_n \vu_n \to \vr \vu \qquad \text { in } L^r(0,T;W^{-1,2}(\Omega;\R^3))
$$
and thus
$$
\vr_n \vu_n \otimes \vu_n \rightharpoonup \vr \vu\otimes \vu \text { in } L^{\frac{r}{r-1}}((0,T)\times\Omega;\R^{3\times 3}).
$$
Exactly as in the previous section we may verify that the limit stress tensor $\tn{T}$ is positive definite. It is based on the estimate of ${\rm Tr}\, [\tn{T}_n]^{-1}$. Then also 
$$
{\rm Tr}\, \log \tn{T}_n \rightharpoonup {\rm Tr}\, \log \tn{T}  \qquad   \text{ in } L^2(0,T; W^{1,2}(\Omega)).
$$
To conclude, we need to show the strong convergence of the density in $L^q((0,T)\times \Omega)$  as well as $\vu_n \to \vu$ in $L^q(0,T;W^{1,q}(\Omega;\R^3))$ for some $q\geq 1$. Here we follow \cite{Feireisl-Liu-Malek}.

Before doing so, with the help of the convergences above we may pass to the limit in the \eqref{en_1}--\eqref{en_4} to conclude that
\begin{equation} \label{enl_1}
\int_0^t \int_{\R^3} \Big(\varrho \pder{\psi}{t} + \vr\vc{u}\cdot \nabla \psi \Big) \dx \,{\rm d}\tau = \int_{\R^3} \vr(t,\cdot)\psi (t,\cdot)\dx -\int_{\R^3} \vr_{0,\Theta} \psi(0,\cdot) \dx
\end{equation}
for all $\psi \in C^\infty_0([0,T]\times \R^3)$, where both the density and velocity are extended by zero outside of $\Omega$,
\begin{equation} \label{enl_2}
\int_0^t \int_\Omega \Big(\eta \pder{\psi}{t} + \eta\vc{u}\cdot \nabla \psi - \varepsilon \nabla \eta \cdot \nabla \psi \Big) \dx \,{\rm d}\tau = \int_\Omega \eta(t,\cdot)\psi (t,\cdot)\dx -\int_\Omega \eta_{0,\Theta} \psi(0,\cdot) \dx 
\end{equation}
for all $\psi \in C^\infty([0,T]\times \overline{\Omega})$,
\begin{equation} \label{enl_3}
\begin{aligned}
&\int_0^t \int_\Omega \Big(\tn{T}: \pder{\tn{M}}{t} + \vc{u}\tn{T}:: \nabla \tn{M} - \varepsilon \nabla \tn{T}::  \nabla \tn{M} \Big) \dx \,{\rm d}\tau  + \int_0^t \int_\Omega (\nabla \vu \tn{T} + \tn{T}\nabla\vu^T):\tn{M}  \dx \,{\rm d}\tau \\
& +\int_0^t \int_\Omega \Big( \frac{k}{2\lambda}(\eta_n +\alpha){\rm Tr}\tn{M} -\frac{1}{2\lambda} \tn{T} :\tn{M} \Big) 
= \int_\Omega \tn{T}(t,\cdot):\tn{M} (t,\cdot)\dx -\int_\Omega \tn{T}_{0,\alpha,\Theta}: \tn{M}(0,\cdot) \dx
\end{aligned}
\end{equation}
for all $\tn{M} \in C^\infty([0,T]\times \overline{\Omega};\R^{3\times 3})$,
\begin{equation} \label{enl_4}
\begin{aligned}
\int_0^t &\int_\Omega \big(\Lambda(\Div \vcg{\varphi}) -\Lambda(\Div\vu)\big)\dx\dtau \geq \frac 12 \int_\Omega \vr |\vu|^2 (t,\cdot) \dx - \int_\Omega \vr_{0,\Theta}|\vu_{0,\Theta}|^2 \dx \\
& - \int_\Omega \vr \vu \cdot \vcg\varphi(t,\cdot)\dx + \int_\Omega \vr_{0,\Theta}\vu_{0,\Theta} \cdot \vcg{\varphi}(0,\cdot)\dx   + \int_0^t \int_\Omega \Big(\vr \vu \cdot \pder{\vcg{\varphi}}{t} + \vr \vu \otimes \vu : \nabla \vcg{\varphi}\Big)\dx\dt \\
&+ \int_0^t\int_\Omega 2\mu_0\Big[\overline{(1+ |\tn{D}^d(\vc{u})|^{2})^{\frac{r-2}{2}}{\tn D}^d(\vu) :{\tn D}^d (\vu)}-\overline{(1+ |\tn{D}^d(\vc{u})|^{2})^{\frac{r-2}{2}}{\tn D}^d(\vu)}:\tn{D}^d(\vcg{\varphi})\Big]  \dx\dtau \\
&+ \int_0^t \int_\Omega(\overline{p(\vr)} \Div \vcg{\varphi}-\overline{p(\vr)\Div \vu)}\dx\dtau + \int_0^t \int_\Omega \Big(\frac {\alpha}{2} {\rm Tr}\, (\log \tn{T}) + kL\eta +\zeta \eta^2  \Big) \Div (\vcg{\varphi}-\vu) \dx\dtau  \\
&+\int_0^t\int_\Omega (\vr \vc{f}\cdot (\vcg{\varphi}-\vu) -\tn{T}:\nabla (\vcg{\varphi}-\vu))\dx\dtau
\end{aligned}
\end{equation}
for all $\vcg{\psi} \in C^\infty_0([0,T]\times \Omega;\R^3)$. 
Evidently, to conclude, it is enough to show that $\tn{D}^d \vu_n \to \tn{D}^d(\vu)$ in $L^{q}((0,T)\times\Omega)$ for some $q>1$  as well as the strong convergence $\vr_n \to \vr$ in some $L^q((0,T)\times \Omega)$, $q\geq 1$. Note also that the term $\overline{(1+ |\tn{D}^d(\vc{u})|^{2})^{\frac{r-2}{2}}{\tn D}^d(\vu) :{\tn D}^d (\vu)}$ is possibly a non-negative measure.

To this aim, exactly as in \cite{Feireisl-Liu-Malek} (we present the proof here for completeness, since we shall use the same argument also in the next sections) we consider as test function in \eqref{enl_4} the function
$$
\vcg{\varphi}_{h,\delta}:= \xi_\delta \mu_{-h}\star \mu_h \star (\xi_\delta \vu)
$$
with $\delta$, $h>0$ and
$$
\mu_h(t):=\frac 1h 1_{[-h,0]}(t), \qquad  \mu_{-h}(t):=\frac 1h 1_{[0,h]}(t)
$$
and $\xi_\delta \in C^\infty_0((0,t))$, $0\leq \xi_\delta \leq 1$, $\xi_\delta(t) =1$ if $t\in [\delta,t-\delta]$, $\delta >0$. Since $\vcg{\varphi}_{h,\delta}$ is a regularized (in time) function $\vu$, the fact that we use this test function in \eqref{enl_4} is in fact replacement testing the inequality by $\vc{u}$ which is not possible to do directly. It is not difficult to verify that
$$
\begin{aligned}
\lim_{\delta \to 0} \lim_{h\to 0} \int_0^t \int_\Omega (\Lambda(\Div \vcg{\varphi}_{h,\delta})-\Lambda(\Div \vu)) \dx\dtau &= 0, \\
\int_\Omega (\vr\vu \cdot \vcg{\varphi}_{h,\delta})(t,\cdot) \dx -\int_\Omega (\vr\vu \cdot \vcg{\varphi}_{h,\delta})(0,\cdot) \dx & =0.
\end{aligned}
$$
Moreover, due to the monotonicity of the stress tensor we have
$$
\big((1+ |\tn{D}^d(\vc{u}_n)|^{2})^{\frac{r-2}{2}}{\tn D}^d(\vu)_n - (1+ |\tn{D}^d(\vc{u})|^{2})^{\frac{r-2}{2}}{\tn D}^d(\vu)\big) :\big({\tn D}^d (\vu_n)-{\tn D}^d (\vu)\big) \geq 0 \qquad \text{ a.e. in } (0,T)\times \Omega;
$$
hence
$$
\begin{aligned}
\int_0^t\int_\Omega \overline{(1+ |\tn{D}^d(\vc{u})|^{2})^{\frac{r-2}{2}}{\tn D}^d(\vu) :{\tn D}^d (\vu)} & \dx\dtau \geq \int_0^t\int_\Omega \overline{(1+ |\tn{D}^d(\vc{u})|^{2})^{\frac{r-2}{2}}{\tn D}^d(\vu)}:\tn{D}^d(\vu)\dx\dtau \\
&=\lim_{\delta \to 0} \lim_{h\to 0} \int_0^t\int_\Omega \overline{(1+ |\tn{D}^d(\vc{u})|^{2})^{\frac{r-2}{2}}{\tn D}^d(\vu)}:\tn{D}^d(\vcg{\varphi}_{h,\delta}) \dx\dtau.
\end{aligned}
$$
The most complicated is the term with the time derivative. We have
$$
\begin{aligned}
\int_0^t \int_\Omega \vr \vu \cdot \pder{\vcg{\varphi_{h,\delta}}}{t}\dx\dtau &= \int_0^t \int_\Omega \Big(\pder{\xi_\delta}{t} \vr \vu   \cdot \big(\mu_{-h} \star \mu_h \star (\xi_\delta \vu)\big)\dx\dtau \\
&+ \int_0^t \int_\Omega \mu_h \star (\vr \vu \xi_\delta)\cdot \Big(\pder{}{t}(\mu_h \star (\xi_\delta \vu))\Big) \dx\dt. 
\end{aligned}
$$
The first term on the right-hand side clearly converges to
$$
\begin{aligned}
\lim_{\delta \to 0} \lim_{h\to 0} & \int_0^t \int_\Omega \Big(\vr \vu  \pder{\xi_\delta}{t} \cdot \mu_{-h} \star \mu_h \star (\xi_\delta \vu)\dx\dtau \\
&=\lim_{\delta \to 0}  \frac 12 \int_0^t \Big(\int_\Omega \vr |\vu|^2 \dx\Big) \pder{\xi_\delta^2}{t}\dtau = -\frac 12 \int_\Omega \vr|\vu|^2 (t,\cdot)\dx + \frac 12 \int_\Omega \vr_0 |\vu_0|^2 \dx.
\end{aligned}
$$
The computations for the second term are more complex. We have (due to the cut-off function we may extend all functions by zero for negative times and time above $t$)
$$
\begin{aligned}
\int_{\R}\int_\Omega  & \mu_h \star (\vr\xi_\delta \vu) \cdot \Big(\pder{}{t} (\mu_h \star (\xi_\delta\vu))\Big) \dx\dt = -\int_\R \int_\Omega  \Big(\pder{}{t} (\mu_h \star (\vr\xi_\delta \vu))\Big) \cdot \big(\mu_h \star (\xi_\delta\vu)\big) \dx\dt \\
& =- \int_\R \int_\Omega \frac{(\vr\xi_\delta \vu)(t+h,x) - (\vr\xi_\delta \vu)(t,x) }{h} \cdot \big(\mu_h \star (\xi_\delta \vu)\big) \dx\dt \\
& + \int_\R \int_\Omega \vr \frac{(\xi_\delta \vu)(t+h,x) - (\xi_\delta \vu)(t,x) }{h} \cdot \big(\mu_h \star (\xi_\delta\vu)\big) \dx\dt 
 -\frac 12 \int_\R \int_\Omega \vr \pder{}{t} \big|\mu_h \star (\xi_\delta\vu) \big|^2 \dx\dt \\
&  =- \int_\R \int_\Omega\Big( \frac{\vr(t+h,x) -\vr(t,x)}{h} (\xi_\delta \vu)(t+h,x) \cdot \big(\mu_h \star (\xi_\delta\vu)\big)  -\frac 12 \vr \pder{}{t} \big|\mu_h \star (\xi_\delta\vu) \big|^2\Big) \dx\dt .
\end{aligned}
$$
Since
$$
\pder{}{t} (\mu_h \star \vr) + \Div (\mu_h \star (\vr\vu)) = 0
$$
in $\mathcal D'(\Omega)$ for any $t\in \R$, where $(\vr,\vu)(t,x) =(\vr_0,\vu_{0})$ for $t<0$ and $(\vr,\vu)(t,x) = (\vr(T),\vc{0})$ for $t>T$, combining this identity with the weak formulation of the continuity equation yields
$$
\begin{aligned}
&\int_{\R}\int_\Omega   \mu_h \star (\vr\xi_\delta \vu) \cdot \Big(\pder{}{t} (\mu_h \star (\xi_\delta\vu))\Big) \dx\dt \\
&= \frac 12 \int_\R \int_\Omega \vr\vu \cdot \nabla |\mu_h \star (\xi_\delta\vu)|^2 \dx\dt - \int_\R \int_\Omega \frac{\vr(t+h,x) -\vr(t,x)}{h} (\xi_\delta \vu)(t+h,x) 
\cdot \big(\mu_h \star (\xi_\delta\vu)\big) \dx\dt     \\
&= \frac 12 \int_\R \int_\Omega \vr\vu \cdot \nabla |\mu_h \star (\xi_\delta\vu)|^2  \dx\dt - \int_\R \int_\Omega \mu_h \star (\vr\vu) \cdot \nabla \big((\xi_\delta \vu)(t+h,x) \cdot \mu_h \star (\xi_\delta\vu)\big) \dx\dt 
.
\end{aligned}
$$
This implies
$$
\lim_{\delta \to 0} \lim_{h\to 0} \int_0^t \int_\Omega \Big(\vr\vu\otimes \vu: \nabla \vcg{\varphi}_{h,\delta} + \big(\mu_h \star (\vr\xi_\delta \vu)\big) \cdot \Big(\pder{}{t} (\mu_h \star (\xi_\delta\vu))\Big)\Big) \dx\dtau = 0.
$$
Whence, letting $h \to 0$ and then $\delta \to 0$ in \eqref{enl_4} with $\vcg{\varphi}:=\vcg{\varphi}_{h,\delta}$ results into
\begin{equation} \label{EVF}
\int_0^t (\overline{p(\vr)}\Div\vu - \overline{p(\vr)\Div \vu}) \dx\dtau \leq 0
\end{equation}
for any $t \in (0,T]$. We show that this inequality (which replaces the effective viscous flux identity here) provides the strong convergence of the sequence of densities. First, recall that the pressure potential $P(\vr)= \vr\psi(\vr)$ is convex and that due to the boundedness of the densities we have both for the sequences $(\vr_n,\vu_n)$ as well as for the limit pair $(\vr,\vu)$ the renormalized continuity equation. Thus, passing with $n \to \infty$, we have both
$$
\begin{aligned}
\pder{P(\vr)}{t} + \Div (P(\vr)\vu) + p(\vr) \Div \vu &= 0, \\
\pder{\overline{P(\vr)}}{t} + \Div (\overline{P(\vr)}\vu) + \overline{p(\vr)\Div \vu} &=0
\end{aligned}
$$
in the weak sense. Since the function $\psi :=1$ is an eligible test function in both cases, we conclude
$$
\int_{\Omega} (\overline{P(\vr)}-P(\vr))(t,\cdot) \dx = -\int_0^t\int_\Omega (\overline{p(\vr)\Div \vu} - p(\vr)\Div \vu) \dx\dtau,
$$
where we used that $\vu=\vc{0}$ on the boundary $\partial \Omega$ and $\overline{P(\vr_0)} = P(\vr_0)$ a.e. in $\Omega$. Due to \eqref{EVF} we have
$$
-\int_0^t \int_\Omega (\overline{p(\vr)\Div \vu} - p(\vr)\Div \vu) \dx\dtau \leq -\int_0^t \int_\Omega (\overline{p(\vr)} -p(\vr))\Div \vu \dx\dtau
$$
and
$$
\begin{aligned}
-&\int_0^t \int_\Omega (\overline{p(\vr)} -p(\vr))\Div \vu \dx\dtau = -\lim_{n\to \infty} \int_0^t \int_\Omega (p(\vr_n)-p(\vr))\Div \vu \dx\dtau \\
&\leq \lim_{n\to \infty}  \int_0^t \int_\Omega p'(\vr) (\vr-\vr_n) \Div \vu \dx\dtau + C \limsup_{n\to \infty} \int_0^t \int_\Omega |\vr_n-\vr|^2\dx \dtau \\
& = C \limsup_{n\to \infty} \int_0^t \int_\Omega |\vr_n-\vr|^2\dx \dtau.
\end{aligned}
$$
Note that here we need the assumption $p\in C^2((0,\infty))$. On the other hand, due to the convexity of the pressure potential
$$
\int_\Omega (\overline{P(\vr)}-P(\vr))\dx \geq c \limsup_{n\to \infty} \int_\Omega |\vr_n-\vr|^2\dx\dtau
$$
for some $c$ positive; the Gronwall argument then concludes that $\vr_n \to \vr$ in $L^2((0,T)\times \Omega)$ and by interpolation, the convergence takes place in any $L^q((0,T)\times \Omega)$, $q<\infty$, and also a.e. in $(0,T)\times \Omega$, at least for a suitably chosen subsequence. Hence $\overline{P(\vr)} = P(\vr)$, $\overline{p(\vr)}=p(\vr)$ and 
$$
\int_0^t \int_\Omega \overline{p(\vr)\Div \vu}\dx\dtau = \int_0^t \int_\Omega p(\vr)\Div \vu\dx\dtau
$$
for any $t\in (0,T]$. Therefore, \eqref{enl_4} with our test function $\vcg{\varphi}_{h,\delta}$ implies
$$
\int_0^t\int_\Omega \Big(\overline{(1+ |\tn{D}^d(\vc{u})|^{2})^{\frac{r-2}{2}}{\tn D}^d(\vu) :{\tn D}^d (\vu)} -\overline{(1+ |\tn{D}^d(\vc{u})|^{2})^{\frac{r-2}{2}}{\tn D}^d(\vu)}:\tn{D}^d(\vu)\Big)\dx\dtau\leq 0.
$$
However, since 
$$
\begin{aligned} 
\int_0^T & \int_\Omega |\tn{D}^d(\vu_n)-\tn{D}^d(\vu)|^2\dx\dtau \\
&\leq \int_0^T \int_\Omega (1+ |\tn{D}^d(\vu_n)|^2)^{\frac{r-2}{2}}\tn{D}^d(\vu_n)-(1+ |\tn{D}^d(\vu)|^2)^{\frac{r-2}{2}}\tn{D}^d(\vu)):(\tn{D}^d(\vu_n)-\tn{D}^d(\vu))\dx\dtau,
\end{aligned}
$$
we get by Korn's inequality $\nabla \vu_n \to \nabla \vu$ in $L^2((0,T)\times \Omega;\R^{3\times 3})$ and thus the convergence also holds in any $L^q((0,T)\times \Omega;\R^{3\times 3})$ for $q<r$. Using also the weak lower semicontinuity argument we finish the proof of the limit passage in the weak formulation. To conclude this part, by the proved convergences as well as the weak lower semicontinuity we can also perform the limit passage in the energy inequality.

\subsection{Estimates independent of $\Theta$, limit passage $\Theta \to 0^+$}
\label{s_4.5}

Since the previous step did not use the improved regularity of the initial condition, we may repeat step by step the arguments from Section \ref{s_4.4} to conclude that after the limit passage $\Theta \to 0^+$ we have existence of the quadruple $(\vr_\alpha,\vu_\alpha,\eta_\alpha,\tn{T}_\alpha)$ such that 

\begin{equation} \label{ealpha_1}
\int_0^t \int_{\R^3} \Big(\varrho_\alpha \pder{\psi}{t} + \vr_\alpha\vc{u}_\alpha\cdot \nabla \psi \Big) \dx \,{\rm d}\tau = \int_{\R^3} \vr_\alpha(t,\cdot)\psi (t,\cdot)\dx -\int_{\R^3} \vr_{0} \psi(0,\cdot) \dx 
\end{equation}
for all $\psi \in C^\infty_0([0,T];\R^3)$ and both the density and velocity are extended by zero outside of $\Omega$,
\begin{equation} \label{ealpha_2}
\int_0^t \int_\Omega \Big(\eta_\alpha \pder{\psi}{t} + \eta_\alpha\vc{u}_\alpha\cdot \nabla \psi - \varepsilon \nabla \eta_\alpha \cdot \nabla \psi \Big) \dx \,{\rm d}\tau = \int_\Omega \eta_\alpha(t,\cdot)\psi (t,\cdot)\dx -\int_\Omega \eta_{0} \psi(0,\cdot) \dx 
\end{equation}
for all $\psi \in C^\infty([0,T]\times\Omega)$,
\begin{equation} \label{ealpha_3}
\begin{aligned}
&\int_0^t \int_\Omega \Big(\tn{T}_\alpha: \pder{\tn{M}}{t} + \vc{u}_\alpha\tn{T}_\alpha:: \nabla \tn{M} - \varepsilon \nabla \tn{T}_\alpha::  \nabla \tn{M} \Big) \dx \,{\rm d}\tau  + \int_0^t \int_\Omega (\nabla \vu_\alpha \tn{T}_\alpha + \tn{T}_\alpha\nabla\vu_{\alpha}^T):\tn{M}  \dx \,{\rm d}\tau \\
& +\int_0^t \int_\Omega \Big( \frac{k}{2\lambda}(\eta_\alpha +\alpha){\rm Tr}\,\tn{M} -\frac{1}{2\lambda} \tn{T}_\alpha :\tn{M} \Big) = \int_\Omega \tn{T}_\alpha(t,\cdot):\tn{M} (t,\cdot)\dx -\int_\Omega \tn{T}_{0, \alpha}: \tn{M}(0,\cdot) \dx
\end{aligned}
\end{equation}
for all $\tn{M} \in C^\infty([0,T]\times \Omega;\R^{3\times 3})$;
above, $\tn{T}_{0,\alpha} = \tn{T}_0 +\alpha \tn{I}$,
\begin{equation} \label{ealpha_4}
\begin{aligned}
\int_0^t &\int_\Omega \big(\Lambda(\Div \vcg{\varphi}) -\Lambda(\Div\vu_\alpha)\big)\dx\dtau \geq \frac 12 \int_\Omega \vr_\alpha |\vu_\alpha|^2 (t,\cdot) \dx - \int_\Omega \vr_{0}|\vu_{0}|^2 \dx \\
& - \int_\Omega \vr_\alpha \vu_\alpha \cdot \vcg\varphi(t,\cdot)\dx + \int_\Omega \vr_{0}\vu_{0} \cdot \vcg{\varphi}(0,\cdot)\dx + \int_0^t \int_\Omega \Big(\vr_\alpha \vu_\alpha \cdot \pder{\vcg{\varphi}}{t} + \vr_\alpha \vu_\alpha \otimes \vu_\alpha : \nabla \vcg{\varphi}\Big)\dx\dt \\
&+ \int_0^t\int_\Omega 2\mu_0(1+ |\tn{D}^d(\vc{u}_\alpha)|^{2})^{\frac{r-2}{2}}{\tn D}^d(\vu_\alpha) :{\tn D}^d (\vu_\alpha-\vcg{\varphi})  \dx\dtau + \int_0^t \int_\Omega p(\vr_\alpha) \Div (\vcg{\varphi}-\vu_\alpha)\dx\dtau \\
&+ \int_0^t \int_\Omega \Big(\frac {\alpha}{2} {\rm Tr}\, (\log \tn{T}_\alpha) + kL\eta_\alpha +\zeta \eta_\alpha^2  \Big) \Div (\vcg{\varphi}-\vu_\alpha) \dx\dtau  \\
&+\int_0^t\int_\Omega (\vr_\alpha \vc{f}\cdot (\vcg{\varphi}-\vu_\alpha) -\tn{T}_\alpha:\nabla (\vcg{\varphi}-\vu_\alpha))\dx\dtau
\end{aligned}
\end{equation}
for all $\vcg{\varphi} \in C^\infty([0,T]\times \Omega;\R^3)$,
together with the energy inequality
\begin{equation} \label{ealpha_5}
\begin{aligned}
& \int_\Omega  \Big( \frac 12 \vr_\alpha|\vu_\alpha|^2 +  P(\vr_\alpha) +  (kL\eta_\alpha \log \eta_\alpha +\zeta\eta_\alpha^2)\Big) (t,\cdot)\dx + \frac 12  \int_\Omega {\rm Tr}\, (\tn{T}_\alpha -\alpha \log (\tn{T}_\alpha))(t,\cdot) \dx  \\
&+ \varepsilon  \int_0^t \int_\Omega \Big(\frac{kL}{\eta_\alpha} + 2\zeta \Big) |\nabla \eta_\alpha|^2 \dx \,{\rm d}\tau +\frac{\alpha k}{4\lambda} \int_0^t\int_\Omega (\eta_\alpha +\alpha){\rm Tr}\, [\tn{T}_\alpha]^{-1} \dx \,{\rm d}\tau  \\
& + \int_0^t\int_\Omega \big(2\mu_0 (1+|\tn{D}^d(\vu_\alpha)|^2)^{\frac{r-2}{2}} |\tn{D}^d(\vu_\alpha)|^2 + \Lambda(\Div \vu_\alpha) \big)\dx \,{\rm d}\tau +\frac{1}{4\lambda}  \int_0^t \int_\Omega {\rm Tr}\, \tn{T}_\alpha\dx\,{\rm d}\tau  \\
& + \frac{\alpha \varepsilon}{6}  \int_0^t \int_\Omega |\nabla {\rm Tr}\, \tn{T}_\alpha|^2 \dx \,{\rm d}\tau \\
& \leq \int_\Omega  \Big( \frac 12 \vr_{0}|\vu_{0}|^2 +  P(\vr_{0}) +  (kL\eta_{0} \log \eta_{0} +\zeta\eta_{0}^2)\Big) \dx  \\
&  + \frac 12  \int_\Omega {\rm Tr}\, (\tn{T}_{0,\alpha} -\alpha \log(\tn{T}_{0,\alpha})) \dx +\int_0^t \int_\Omega \vr_\alpha\vc{f}\cdot \vu_\alpha \dx\,{\rm d}\tau + \frac{3k}{4\lambda} \int_0^t \int_\Omega (\eta_\alpha+\alpha) \dx \,{\rm d}\tau + \frac{3\alpha T}{4\lambda} |\Omega|.
\end{aligned}
\end{equation}

\subsection{Estimates independent of $\alpha$, limit passage $\alpha \to 0^+$}
\label{s_4.6}

We finally need to pass with $\alpha \to 0$. Since the procedure is almost identical to that from Section \ref{s_4.4}, we will concentrate only on the terms with explicit dependence on $\alpha$. First of all, since $\tn{T}_{0,\alpha} = \tn{T}_0 +\alpha \tn{I}$, we easily see that
$$
\alpha \int_\Omega {\rm Tr} \log \tn T_{0,\alpha} \dx \to 0 \qquad \text{ for } \alpha \to 0^+
$$  
as well as
$$
\int_\Omega {\rm Tr}\, \tn{T}_{0,\alpha} \dx \to \int_\Omega {\rm Tr}\, \tn{T}_0 \qquad \text{ for } \alpha \to 0^+.
$$
Therefore the right-hand side of the energy inequality \eqref{ealpha_5} is bounded for $\alpha \to 0^+$ and we get the same estimates and convergence as above. We now look at the terms on the left-hand side of the energy inequality. We clearly have (always for $\alpha >0$ or $\lim_{\alpha \to 0^+}$)
$$
\begin{aligned}
\lim_{\alpha \to 0^+} \int_\Omega ({\rm Tr}\, (\tn{T}_\alpha -\alpha \log(\tn{T_\alpha})) \dx &\geq \int_\Omega {\rm Tr}\, \tn{T}\dx \\
\int_0^t\int_\Omega (\eta_\alpha +\alpha) {\rm Tr}\, ([\tn{T}]^{-1}) \dx\dtau &\geq 0 \\
\int_0^t\int_\Omega  {\rm Tr}\, \tn{T}_\alpha\dx\dtau &\to \int_0^t\int_\Omega  {\rm Tr}\, \tn{T} \dx\dtau \\
\alpha \int_0^t \int_\Omega |\nabla {\rm Tr}\, \tn{T}_\alpha|^2 \dx\dtau \geq 0.
\end{aligned}
$$
Thus, we may pass to the limit in \eqref{ealpha_5} to get (we use either the corresponding convergences or the weak lower semicontinuity)
\begin{equation} \label{limit_5}
\begin{aligned}
& \int_\Omega  \Big( \frac 12 \vr|\vu|^2 +  P(\vr_\alpha) +  (kL\eta \log \eta +\zeta\eta^2)\Big) (t,\cdot)\dx + \frac 12  \int_\Omega {\rm Tr}\, \tn{T}(t,\cdot) \dx  \\
&+ \varepsilon  \int_0^t \int_\Omega \Big(\frac{kL}{\eta} + 2\zeta \Big) |\nabla \eta|^2 \dx \,{\rm d}\tau   \\
& + \int_0^t\int_\Omega \big(2\mu_0 (1+|\tn{D}^d(\vu)|^2)^{\frac{r-2}{2}} |\tn{D}^d(\vu)|^2 + \Lambda(\Div \vu) \big)\dx \,{\rm d}\tau +\frac{1}{4\lambda}  \int_0^t \int_\Omega {\rm Tr}\, \tn{T}\dx\,{\rm d}\tau  \\
& \leq \int_\Omega  \Big( \frac 12 \vr_{0}|\vu_{0}|^2 +  P(\vr_{0}) +  (kL\eta_{0} \log \eta_{0} +\zeta\eta_{0}^2)\Big) \dx  \\
&  + \frac 12  \int_\Omega {\rm Tr}\, \tn{T}_{0} \dx +\int_0^t \int_\Omega \vr\vc{f}\cdot \vu \dx\,{\rm d}\tau + \frac{3k}{4\lambda} \int_0^t \int_\Omega \eta \dx \,{\rm d}\tau.
\end{aligned}
\end{equation}

To conclude, we repeat the procedure from Section \ref{s_4.4}, in particular the proof of the strong convergence of velocity gradient and density. The only slightly more difficult term to treat is
$$
\int_0^t \int_\Omega \frac{\alpha}{2} {\rm Tr}\, (\log \tn{T}_\alpha) \Div (\varphi-\vu_\alpha)\dx\dtau.
$$
However, due to the control of
$$
\alpha \int_0^t \int_\Omega |\nabla ({\rm Tr}\, (\log \tn{T}_\alpha))|^2 \dx\dtau \leq C
$$
and since $\|\vcg{\varphi}-\vu_\alpha\|_{L^2((0,T)\times \Omega)}\leq C$, we easily get
$$
\begin{aligned}
\int_0^t \int_\Omega & \frac{\alpha}{2} {\rm Tr}\, (\log \tn{T}_\alpha) \Div (\varphi-\vu_\alpha)\dx\dtau \\
&\leq \sqrt{\alpha}\sqrt{\alpha} \Big(\int_0^t \int_\Omega |\nabla ({\rm Tr}\, (\log \tn{T}_\alpha))|^2 \dx\dtau \Big)^{\frac 12} \Big(\int_0^t\int_\Omega |\vu_\alpha-\vcg{\varphi}|^2\dx\dtau\Big)^{\frac 12} \to 0
\end{aligned}
$$ 
for $\alpha \to 0^+$. We therefore obtain exactly the weak formulation from Definition \ref{weak_solution}. The proof of Theorem \ref{main_1} is complete.

\section{Weak-strong uniqueness}   
\label{s_5}

As already mentioned, in this part we follow paper \cite{Lu-Z.Zhang18}, however, certain modifications will be needed. We consider two solutions corresponding to the same data (initial conditions, right-hand side and boundary conditions), the weak solution in the sense of Definition \ref{weak_solution} (constructed, e.g., by the methods from the previous section) $(\vr,\vu,\tn{T}, \eta)$ and a strong solution $(\tvr,\tvu,\tnT,\teta)$; here, however, we cannot guarantee that such a solution with the given regularity as in Theorem \ref{main_2} exists, even for small data or even on a short time interval. The aim is to show that under the assumptions from Theorem \ref{main_2} the solutions coincide.

Recall that by the form of the stress tensor, the velocity divergence is bounded and thus also the density is bounded both from zero and from infinity since the initial density is so. Similarly also the quantity $\eta$ is bounded away from zero. More precisely, since
$$
0 < \underline{\vr_0} \leq \vr_0(x) \leq \overline{\vr_0} <\infty, \qquad 0 < \underline{\eta_0} \leq \eta_0(x),
$$
we have
$$
0 < C_1(T)\underline{\vr_0} \leq \vr(t,x), \tvr(t,x) \leq C_2(T) \overline{\vr_0} \qquad 0 < C_1(T) \underline{\eta_0} \leq \eta(t,x), \teta(t,x)
$$
almost everywhere (or everywhere) on $(0,T)\times \Omega$. 

Recall that 
$$
p(\vr) = \vr^2 \psi(\vr),  \qquad q(\eta) = kL \eta + \zeta \eta^2
$$
and we set
$$
\begin{aligned}
H(\vr) = \vr\psi(\vr), \quad &\text{ thus } H(\vr) = \vr H'(\vr) -H(\vr) \\
G(\eta) = kL\eta \log\eta + \zeta \eta^2, \quad &\text{ thus } q(\eta) = \eta G'(\eta) -G(\eta).
\end{aligned}
$$
Recall that both $H(z)$ and $G(s)$ are convex functions on $(0,\infty)$. We further define
\begin{equation} \label{def_E}
\begin{aligned}
\mathcal E_1(t) & = \mathcal E_1[(\vr,\vu),(\tvr,\tvu)](t) := \int_\Omega \Big(\frac 12 \vr |\vu-\tvu|^2 + H(\vr) -H(\tvr) -H'(\tvr) (\vr-\tvr)\Big) (t,\cdot)\dx \\
\mathcal E_2(t) & = \mathcal E_2[\eta,\teta](t) := \int_\Omega \big(G(\eta) -G(\teta) -G'(\teta) (\eta-\teta)\big) (t,\cdot)\dx.
\end{aligned}
\end{equation} 

Since both densities and also $\eta$ and $\teta$ are bounded away from zero and $G$ and $H$ are both convex, we have 
\begin{lemma} \label{convexity}
Under our assumptions from Theorem \ref{main_2}, we have
\begin{equation} \label{quadratic}
\begin{aligned}
H(\vr) -H(\tvr) -H'(\tvr) (\vr-\tvr) &\geq C(T,\overline{\vr_0})|\vr-\tvr|^2\\
G(\eta) -G(\teta) -G'(\teta) (\eta-\teta) &\geq C(T,\overline{\eta_0})|\eta-\teta|^2. 
\end{aligned}
\end{equation}
\end{lemma}
\begin{proof}
The proof follows from \cite[Lemma 5.2]{Lu-Z.Zhang18}. We used the fact that all quantities $\vr$, $\tvr$, $\eta$, $\teta$ are  strictly  positive.
\end{proof}  

We now start with the computations of the main inequality, often called the relative entropy (or energy) inequality. For a moment we only assume that $(\vr,\vu,\tn{T},\eta)$ is a weak solution to our problem \eqref{01a}--\eqref{04a} corresponding to the initial conditions $(\vr_0,\vu_0,\tn{T}_0,\eta_0)$ (satisfying assumptions of Theorem \ref{main_2}) and $(\tvr,\tvu,\tnT,\teta)$ are sufficiently regular functions and only the homogeneous Dirichlet conditions for $\tvu$ must be assumed now. The other assumptions on these functions will be used later.

Since $\eta$ is a renormalized   to \eqref{03a}, we have
\begin{equation} \label{5.1} 
\int_\Omega G(\eta)(t,\cdot)\dx = \int_\Omega G(\eta_0)\dx - \int_0^t \int_\Omega q(\eta) \Div\vu \dx\dtau -2\varepsilon \int_0^t \int_\Omega (2kL |\nabla \eta^{\frac 12}|^2 + \zeta |\nabla \eta|^2)\dx\dtau.
\end{equation}
Similarly for the density we have (recall that the density is a bounded function and thus the renormalized equation holds)
\begin{equation} \label{5.3} 
\int_\Omega H(\vr)(t,\cdot)\dx = \int_\Omega H(\vr_0)\dx - \int_0^t \int_\Omega p(\vr) \Div\vu \dx\dtau.
\end{equation}

We further use as a test function in the weak formulation of \eqref{01a} the function $H'(\tvr)$. It results into (by our assumptions in Theorem \ref{main_2} it is an eligible test function)
\begin{equation} \label{5.4}
\int_\Omega \vr H'(\tvr) (t,\cdot)\dx = \int_\Omega \vr_0 H'(\tvr_0) \dx + \int_0^t\int_\Omega \Big(\vr \pder{H'(\tvr)}{t} + \vr\vu \cdot \nabla H'(\tvr)\Big) \dx\dtau.
\end{equation}
We can also use as a test function in the weak formulation of \eqref{01a} the function $\frac 12 |\tvu|^2$ and have
\begin{equation} \label{5.5} 
\int_\Omega \frac 12 \vr |\tvu|^2 (t,\cdot) \dx = \int_\Omega \frac 12 \vr_0 |\tvu_0|^2 \dx + \int_0^t \int_\Omega \Big(\vr\tvu \cdot \pder{\tvu}{t} + \vr(\vu \cdot \nabla \tvu)\cdot \tvu\Big) \dx\dtau.
\end{equation}
Next identity we use is the weak formulation of \eqref{02a} with the test function $\tvu$. Then
\begin{equation} \label{5.6}
\begin{aligned} 
&\int_\Omega \vr\vu \cdot \tvu (t,\cdot)\dx \geq \int_\Omega \vr_0\vu_0\cdot \tvu_0 \dx + \frac 12 \int_\Omega \vr|\vu|^2 (t,\cdot)\dx -\frac 12 \int_\Omega \vr_0 |\vu_0|^2\dx \\
& + \int_0^t\int_\Omega (\Lambda (\Div\vu) - \Lambda (\Div \tvu))\dx\dtau + \int_0^t \int_\Omega \Big(\vr\vu \cdot \pder{\tvu}{t} + \vr(\vu \otimes\vu) :\nabla \tvu\Big) \dx\dtau \\
& + \int_0^t \int_\Omega \Big(2\mu_0(1+ |\tn{D}^d(\vu)|^2)^{\frac{r-2}{2}} \tn{D}^d(\vu-\tvu) + p(\vr) \Div (\tvu-\vu)\Big) \dx\dtau \\
& + \int_0^t \int_\Omega \Big(\vr \vc{f}\cdot (\tvu-\vu) - \tn{T} :\nabla (\tvu-\vu) + q(\eta) \Div(\tvu-\vu)\Big)\dx\dtau.
\end{aligned}
\end{equation}
Using $\frac 12 \vr|\vu-\tvu|^2 = \frac 12 \vr|\vu|^2 - \vr\vu\cdot \tvu + \frac 12 \vr|\tvu|^2$ and and taking \eqref{5.1} + \eqref{5.3} $-$\eqref{5.4} + \eqref{5.5} $-$ \eqref{5.6} we end up with 
\begin{equation} \label{5.7}
\begin{aligned}
&\int_\Omega \Big( \frac 12 \vr|\vu-\tvu|^2 + H(\vr)-H'(\tvr)\vr + G(\eta)\Big)(t,\cdot)\dx +  \int_0^t\int_\Omega (\Lambda (\Div\vu) - \Lambda (\Div \tvu))\dx\dtau \\
& + \int_0^t \int_\Omega 2\mu_0(1+ |\tn{D}^d(\vu)|^2)^{\frac{r-2}{2}} \tn{D}^d(\vu):\tn{D}^d(\vu-\tvu) \dx\dtau + 2\varepsilon \int_0^t \int_\Omega (2kL |\nabla \eta^{\frac 12}|^2 + \zeta |\nabla \eta|^2)\dx\dtau \\
&\leq \int_\Omega \Big( \frac 12 \vr_0|\vu_0-\tvu_0|^2 + H(\vr_0)-H'(\tvr_0)\vr_0 + G(\eta_0)\Big)\dx +\int_0^t \int_\Omega (\tn{T}: \nabla(\tvu -\vu) -q(\eta) \Div \tvu)\dx\dtau \\
& + \int_0^t \int_\Omega \Big(\vr\vc{f} \cdot (\vu-\tvu) -p(\vr)\Div \tvu-\vr \pder{H'(\tvr)}{t}-\vr\vu \cdot\nabla H'(\tvr) \Big)\dx\dtau \\
& + \int_0^t\int_\Omega \Big( \vr\tvu\cdot\pder{\tvu}{t} + \vr(\vu \cdot \nabla \tvu)\cdot \tvu -\vr\vu \cdot \pder{\tvu}{t} -\vr(\vu\otimes \vu):\nabla \tvu\Big)\dx\dtau.
\end{aligned}
\end{equation}
We now compute 
$$
\vr \pder{H'(\tvr)}{t}+ \vr\vu \cdot \nabla H'(\tvr) = \tvr \pder{H'(\tvr)}{t} + \tvr \tvu \cdot \nabla H'(\tvr) +(\vr-\tvr)  \pder{H'(\tvr)}{t} + (\vr\vu-\tvr\tvu)\cdot \nabla H'(\tvr) 
$$
and using
$$
p(\tvr) = \tvr H'(\tvr) -H(\tvr)
$$
we get
\begin{equation} \label{5.8}
\vr \pder{H'(\tvr)}{t}+ \vr\vu \cdot \nabla H'(\tvr) = (\vr-\tvr)  \pder{H'(\tvr)}{t} + (\vr\vu-\tvr\tvu)\cdot \nabla H'(\tvr) + \pder{p(\tvr)}{t} + \Div (p(\tvr)\tvu) -p(\tvr)\Div\tvu.
\end{equation}
We plug in \eqref{5.8} into \eqref{5.7} and get by straightforward computations
\begin{equation} \label{5.9}
\begin{aligned}
&\int_\Omega \Big( \frac 12 \vr|\vu-\tvu|^2 + H(\vr)-H(\tvr)- (\vr-\tvr) H'(\tvr) + G(\eta)\Big)(t,\cdot)\dx \\
&+  \int_0^t\int_\Omega \big(\Lambda (\Div\vu) - \Lambda (\Div \tvu) - \Lambda'(\Div\tvu) (\Div \vu-\Div \tvu)\big)\dx\dtau \\
& + \int_0^t \int_\Omega 2\mu_0\Big[(1+ |\tn{D}^d(\vu)|^2)^{\frac{r-2}{2}} \tn{D}^d(\vu)- (1+ |\tn{D}^d(\tvu)|^2)^{\frac{r-2}{2}}\tn{D}^d(\tvu)\Big]:(\tn{D}^d(\vu)-\tn{D}^d(\tvu)) \dx\dtau \\
&+ 2\varepsilon \int_0^t \int_\Omega (2kL |\nabla \eta^{\frac 12}|^2 + \zeta |\nabla \eta|^2)\dx\dtau \\
&\leq \int_\Omega \Big( \frac 12 \vr_0|\vu_0-\tvu_0|^2 + H(\vr_0)-H(\tvr_0) -(\vr_0-\tvr_0)H'(\tvr_0) + G(\eta_0)\Big)\dx \\
& + \int_0^t \int_\Omega \Big(\Lambda'(\Div\tvu) (\Div \tvu-\Div \vu) + 2\mu_0 (1+ |\tn{D}^d(\tvu)|^2)^{\frac{r-2}{2}}\tn{D}^d(\tvu):(\tn{D}^d(\tvu)-\tn{D}^d(\vu))\Big)\dx\dtau \\
&-\int_0^t \int_\Omega \Big((\vr-\tvr)  \pder{H'(\tvr)}{t} + (\vr\vu-\tvr\tvu)\cdot \nabla H'(\tvr)-(p(\tvr)-p(\vr))\Div\tvu\Big)\dx\dtau \\
&+\int_0^t \int_\Omega (\tn{T}: \nabla(\tvu -\vu) -q(\eta) \Div \tvu)\dx\dtau \\
& + \int_0^t \int_\Omega \Big[\vr\vc{f} \cdot (\vu-\tvu) + \vr(\tvu-\vu)\cdot \Big( \pder{\tvu}{t} +\vu \cdot \nabla \tvu\Big)\Big]\dx\dtau.
\end{aligned}
\end{equation}
Exactly as for \eqref{5.8}, we get 
\begin{equation} \label{5.10}
\eta \pder{G'(\teta)}{t}+ \vr\vu \cdot \nabla G'(\teta) = (\eta-\teta)  \pder{G'(\teta)}{t} + (\eta\vu-\teta\tvu)\cdot \nabla G'(\teta) + \pder{q(\teta)}{t} + \Div (q(\teta)\tvu) -q(\teta)\Div\tvu.
\end{equation}
Using as test function in the weak formulation of \eqref{03a} the function $G'(\teta)$ we get
\begin{equation} \label{5.11}
\int_\Omega \eta G'(\teta)(t,\cdot)\dx = \int_\Omega \eta_0 G'(\teta_0) \dx + \int_0^t \int_\Omega \Big[ \eta \pder{G'(\teta)}{t} + \eta \vu \cdot \nabla G'(\teta) -\varepsilon \nabla \eta \cdot\nabla G'(\teta)\Big] \dx\dtau.
\end{equation}
Summing up \eqref{5.9} (integrated over $\Omega$) with \eqref{5.10} and \eqref{5.11} (multiplied by $-1$) yields
\begin{equation} \label{5.12}
\begin{aligned}
&\int_\Omega \Big( \frac 12 \vr|\vu-\tvu|^2 + H(\vr)-H(\tvr)- (\vr-\tvr) H'(\tvr) + G(\eta)-G(\teta)-G'(\teta)(\eta-\teta) \Big)(t,\cdot)\dx \\
&+  \int_0^t\int_\Omega \big(\Lambda (\Div\vu) - \Lambda (\Div \tvu) - \Lambda'(\Div\tvu) (\Div \vu-\Div \tvu)\big)\dx\dtau \\
& + \int_0^t \int_\Omega 2\mu_0\Big[(1+ |\tn{D}^d(\vu)|^2)^{\frac{r-2}{2}} \tn{D}^d(\vu)- (1+ |\tn{D}^d(\tvu)|^2)^{\frac{r-2}{2}}\tn{D}^d(\tvu)\Big]:(\tn{D}^d(\vu)-\tn{D}^d(\tvu)) \dx\dtau \\
&+ 2\varepsilon \int_0^t \int_\Omega (2kL |\nabla \eta^{\frac 12}|^2 + \zeta |\nabla \eta|^2)\dx\dtau \\
&\leq \int_\Omega \Big( \frac 12 \vr_0|\vu_0-\tvu_0|^2 + H(\vr_0)-H(\tvr_0) -(\vr_0-\tvr_0)H'(\tvr_0) + G(\eta_0)-G(\teta_0)-G'(\teta_0)(\eta_0-\teta_0) \Big)\dx \\
& + \int_0^t \int_\Omega \Big(\Lambda'(\Div\tvu) (\Div \tvu-\Div \vu) + 2\mu_0 (1+ |\tn{D}^d(\tvu)|^2)^{\frac{r-2}{2}}\tn{D}^d(\tvu):(\tn{D}^d(\tvu)-\tn{D}^d(\vu))\Big)\dx\dtau \\
&-\int_0^t \int_\Omega \Big((\vr-\tvr)  \pder{H'(\tvr)}{t} + (\vr\vu-\tvr\tvu)\cdot \nabla H'(\tvr)-(p(\tvr)-p(\vr))\Div\tvu\Big)\dx\dtau \\
&-\int_0^t \int_\Omega \Big((\eta-\teta)  \pder{G'(\teta)}{t} + (\eta\vu-\teta\tvu)\cdot \nabla G'(\teta)-(q(\teta)-q(\eta))\Div\tvu\Big)\dx\dtau \\
&+\int_0^t \int_\Omega (\tn{T}: \nabla(\tvu -\vu) + \varepsilon \nabla \eta \cdot \nabla G'(\teta))\dx\dtau \\
& + \int_0^t \int_\Omega \Big[\vr\vc{f} \cdot (\vu-\tvu) + \vr(\tvu-\vu)\cdot \Big( \pder{\tvu}{t} +\vu \cdot \nabla \tvu\Big)\Big]\dx\dtau.
\end{aligned}
\end{equation}
Since 
$$
\nabla \eta \cdot \nabla G'(\teta) = \nabla \eta (kL \teta^{-1} + 2\zeta)\cdot \nabla \teta,
$$
and
$$
\begin{aligned}
4 |\nabla \eta^{\frac 12}|^2 - \teta^{-1}\nabla\eta \cdot \nabla \teta &= 4 \big| \nabla(\eta^{\frac 12} -\teta^{\frac 12})\big|^2 +4 \nabla \teta^{\frac 12} \cdot \nabla (\eta^{\frac 12} -\teta^{\frac 12}) 
+ 4 \nabla \eta^{\frac 12} \cdot \nabla \teta^{\frac 12} (1-\eta^{\frac 12}\teta^{\frac 12}) \\
|\nabla \eta|^2 -\nabla \eta \cdot \nabla \teta & = |\nabla (\eta-\teta)|^2 + \nabla \teta \cdot \nabla (\eta-\teta),
\end{aligned}
$$
we have
\begin{equation} \label{5.13}
\begin{aligned}
&\mathcal E_1[(\vr,\vu),(\tvr,\tvu)](t) + \mathcal E_2[\eta,\teta](t) +  \int_0^t\int_\Omega (\Lambda (\Div\vu) - \Lambda (\Div \tvu) - \Lambda'(\Div\tvu) (\Div \vu-\Div \tvu))\dx\dtau \\
& + \int_0^t \int_\Omega 2\mu_0\Big[(1+ |\tn{D}^d(\vu)|^2)^{\frac{r-2}{2}} \tn{D}^d(\vu)- (1+ |\tn{D}^d(\tvu)|^2)^{\frac{r-2}{2}}\tn{D}^d(\tvu)\Big]:(\tn{D}^d(\vu)-\tn{D}^d(\tvu)) \dx\dtau \\
& + 2\varepsilon \int_0^t \int_\Omega \Big(2kL \big|\nabla (\eta^{\frac 12}-\teta^{\frac 12})\big|^2 + \zeta |\nabla (\eta-\teta)|^2\Big)\dx\dtau \\
& \leq \mathcal E_1[(\vr_0,\vu_0),(\tvr_0,\tvu_0)] + \mathcal E_2[\eta_0,\teta_0] + \int_0^t \mathcal R \dtau,
\end{aligned}
\end{equation}
where $\mathcal R(\tau) = \sum_{j=1}^5 \mathcal R_j (\tau)$ and
$$
\begin{aligned}
\mathcal R_1(\tau) &= \int_\Omega \Big[\vr\vc{f} \cdot (\vu-\tvu) + \vr(\tvu-\vu)\cdot \Big( \pder{\tvu}{t} +\vu \cdot \nabla \tvu\Big)\Big](\tau,\cdot)\dx \\
&+  \int_\Omega 2\mu_0 (1+ |\tn{D}^d(\tvu)|^2)^{\frac{r-2}{2}}\tn{D}^d(\tvu):(\tn{D}^d(\tvu)-\tn{D}^d(\vu))(\tau,\cdot) \dx \\
& + \int_\Omega \Big(\Lambda'(\Div\tvu) (\Div \tvu-\Div \vu) + (\tvr-\vr)  \pder{H'(\tvr)}{t} + (\tvr\tvu-\vr\vu)\cdot \nabla H'(\tvr)\Big)(\tau,\cdot)\dx \\
& + \int_\Omega (p(\tvr)-p(\vr)) \Div \tvu (\tau,\cdot) \dx \\
\mathcal R_2(\tau) & = \int_\Omega \Big((\teta-\eta)  \pder{G'(\teta)}{t} + (\teta\tvu-\eta\vu)\cdot \nabla G'(\teta)+(q(\teta)-q(\eta))\Div\tvu\Big)(\tau,\cdot)\dx \\
\mathcal R_3(\tau) & = -4\varepsilon kL \int_\Omega \Big[ \nabla \teta^{\frac 12} \cdot \nabla (\eta^{\frac 12} -\teta^{\frac 12}) + \nabla \eta^{\frac 12} \cdot \nabla \teta^{\frac 12} (1-\eta^{\frac 12}\teta^{-\frac 12}) \Big](\tau,\cdot)\dx \\
\mathcal R_4(\tau) &= -2\zeta\int_\Omega \nabla \teta \cdot \nabla (\eta-\teta)(\tau,\cdot)\dx \\
\mathcal R_5(\tau) &= \int_\Omega \tn{T}: \nabla (\tvu-\vu)(\tau,\cdot)\dx.
\end{aligned}
$$
We now assume that $(\tvr,\tvu,\teta,\tnT)$ is a strong solution to our problem \eqref{01a}--\eqref{04a} corresponding to the same data. Therefore $\mathcal E_1[(\vr_0,\vu_0),(\tvr_0,\tvu_0)] + \mathcal E_2[\eta_0,\teta_0] =0$ and we have
$$
\begin{aligned}
&\pder{\tvu}{t} + \tvu \cdot \nabla \tvu + \tvr^{-1} \nabla p(\tvr) -\tvr^{-1} \nabla\Lambda(\Div \tvu) -\tvr^{-1} 2\mu_0 \Div ((1+|\tn{D}^d(\tvu)|^2)^{\frac {r-2}{2}}\tn{D}^d(\tvu)) \\
&= \tvr^{-1} \Div \tnT -\tvr^{-1} \nabla q(\teta) + \vc{f}.
\end{aligned}
$$
Plugging this relation into $\mathcal R_1$ together with $\tvr^{-1} \nabla p(\tvr) =\nabla H'(\tvr)$, we end up with  
$$
\begin{aligned}
\mathcal R_1(\tau) &= \int_\Omega  \vr((\vu-\tvu)\cdot \nabla \tvu)\cdot (\tvu-\vu) (\tau,\cdot)\dx \\
 &+ 2\mu_0 \int_\Omega \Div((1+ |\tn{D}^d(\tvu)|^2)^{\frac{r-2}{2}}\tn{D}^d(\tvu))\cdot (\tvu-\vu) (\vr \tvr^{-1}-1) (\tau,\cdot)\dx \\
& + \int_\Omega \nabla \Lambda'(\Div\tvu) \cdot ( \tvu- \vu) (\vr \tvr^{-1}-1) (\tau,\cdot)\dx \\
& -\int_\Omega \Big[ \vr \nabla H'(\tvr) \cdot (\tvu-\vu) - (\tvr-\vr)  \pder{H'(\tvr)}{t} - (\tvr\tvu-\vr\vu)\cdot \nabla H'(\tvr)\Big](\tau,\cdot)\dx \\
& + \int_\Omega \Big[ \vr\tvr^{-1} \Div \tnT \cdot (\tvu-\vu) - \vr\tvr^{-1} \nabla q(\teta)\cdot (\tvu-\vu) + (p(\tvr)-p(\vr)) \Div \tvu\Big] (\tau,\cdot) \dx
\end{aligned}
$$
We now use that $\tvr$ is a renormalized solution to \eqref{01a} and thus we may rewrite the fourth term to the form
$$
\begin{aligned}
-\vr \nabla H'(\tvr) \cdot &(\tvu-\vu)   + (\tvr-\vr)  \pder{H'(\tvr)}{t} + (\tvr\tvu-\vr\vu)\cdot \nabla H'(\tvr) \\
&= (\tvr-\vr) \Big[ \pder{H'(\tvr)}{t} + \Div (H'(\tvr) \tvu) + (H''(\tvr)\tvr-H'(\tvr))\Div \tvu\Big] -(\tvr-\vr) H''(\tvr)\tvr \Div \tvu \\
& =-(\tvr-\vr) p'(\tvr) \Div \tvu.
\end{aligned}
$$
Therefore we may write $\mathcal R_1(\tau) = \sum_{j=1}^7 \mathcal R_1^j (\tau)$, where 
$$
\begin{aligned}
\mathcal R_1^1(\tau) &= \int_\Omega  \vr((\vu-\tvu)\cdot \nabla \tvu)\cdot (\tvu-\vu) (\tau,\cdot)\dx \\
\mathcal R_1^2(\tau) &= 2\mu_0 \int_\Omega \Div((1+ |\tn{D}^d(\tvu)|^2)^{\frac{r-2}{2}}\tn{D}^d(\tvu))\cdot (\tvu-\vu) (\vr \tvr^{-1}-1) (\tau,\cdot)\dx \\
\mathcal R_1^3(\tau) &= \int_\Omega \nabla \Lambda'(\Div\tvu) \cdot( \tvu-\vu) (\vr \tvr^{-1}-1) (\tau,\cdot)\dx \\
\mathcal R_1^4(\tau) &= \int_\Omega \Div \tvu (p(\tvr)-p(\vr) -p'(\tvr) (\tvr-\vr)) (\tau,\cdot)\dx \\
\mathcal R_1^5(\tau) &= \int_\Omega (1-\vr\tvr^{-1}) \nabla q(\teta)\cdot (\tvu-\vu) (\tau,\cdot)\dx + \int_\Omega (\vr\tvr^{-1}-1) \Div \tnT \cdot (\tvu-\vu)(\tau,\cdot)\dx \\
\mathcal R_1^6(\tau) &= \int_\Omega \teta \nabla G'(\teta) \cdot (\vu-\tvu) (\tau,\cdot)\dx \\ 
\mathcal R_1^7(\tau) & = \int_\Omega \tnT: \nabla (\vu-\tvu) (\tau,\cdot)\dx.
\end{aligned}
$$
We clearly have 
$$
\mathcal R_1^7 (\tau) + \mathcal R_5(\tau) = \int_\Omega (\tnT-\tn{T}): \nabla (\vu-\tvu) (\tau,\cdot)\dx.
$$
Next we combine $\mathcal R_2(\tau)$ and $\mathcal R_1^6(\tau)$. Similarly ar for the density, we use that $\teta$ is renormalized solution to the corresponding equation \eqref{03a} and we also use the relation between $G$ and $q$. We get
$$
\begin{aligned}
&(\teta-\eta) \pder{G'(\teta)}{t} + (\teta\tvu-\eta\vu)\cdot \nabla G'(\teta) - \teta (\tvu-\vu) \cdot \nabla G'(\teta) + (q(\teta)-q(\eta)) \Div \tvu \\
& = (\teta-\eta) \Big[\pder{G'(\teta)}{t} + \Div( G'(\teta) \tvu) + (G''(\teta)\teta-G'(\teta)) \Div \tvu -\varepsilon \Delta \teta G''(\teta)\Big] \\
&\qquad \qquad -(\teta-\eta)(\tvu-\vu)\cdot \nabla G'(\teta) -(\teta-\eta) G''(\teta)\teta \Div \tvu + \varepsilon \Delta \teta G''(\teta) (\teta-\eta) + (q(\teta)-q(\eta)) \Div \tvu \\
&= \varepsilon \Delta \teta G''(\teta) (\teta-\eta)+ \Div \tvu (q(\teta)-q(\eta) -q'(\teta) (\teta-\eta)) -(\teta-\eta)(\tvu-\vu)\cdot \nabla G'(\teta).
\end{aligned}
$$
Thus
$$
\begin{aligned}
\mathcal R_2(\tau) +\mathcal R_1^6(\tau) &= \int_\Omega \varepsilon \teta^{-1} \Delta \teta q'(\teta) (\teta-\eta)(\tau,\cdot)\dx \\
& + \int_\Omega \Big[\Div \tvu (q(\teta)-q(\eta) -q'(\teta) (\teta-\eta)) - (\teta-\eta)(\tvu-\vu)\cdot \nabla G'(\teta)\Big] (\tau,\cdot)\dx.
\end{aligned}
$$
Altogether, we get
$$
\begin{aligned}
\mathcal R(\tau) &= \int_\Omega  \vr((\vu-\tvu)\cdot \nabla \tvu)\cdot (\tvu-\vu) (\tau,\cdot)\dx  \\
&+ 2\mu_0 \int_\Omega \Div((1+ |\tn{D}^d(\tvu)|^2)^{\frac{r-2}{2}}\tn{D}^d(\tvu))\cdot (\tvu-\vu) (\vr \tvr^{-1}-1) (\tau,\cdot)\dx \\
& + \int_\Omega \nabla \Lambda'(\Div\tvu) \cdot( \tvu- \vu) (\vr \tvr^{-1}-1) (\tau,\cdot)\dx  - \int_\Omega (\teta-\eta)(\tvu-\vu)\cdot \nabla G'(\teta) (\tau,\cdot)\dx \\
& + \int_\Omega \Div \tvu (p(\tvr)-p(\vr) -p'(\tvr) (\tvr-\vr)) (\tau,\cdot)\dx + \int_\Omega \Div \tvu (q(\teta)-q(\eta) -q'(\teta) (\teta-\eta)) (\tau,\cdot)\dx \\
&+ \int_\Omega (1-\vr\tvr^{-1}) \nabla q(\teta)\cdot (\tvu-\vu) (\tau,\cdot)\dx + \int_\Omega (\vr\tvr^{-1}-1) \Div \tnT \cdot (\tvu-\vu)(\tau,\cdot)\dx \\
&+ \varepsilon kL \int_\Omega \Big[ 4 \teta^{-\frac 12} (\teta^{\frac 12}-\eta^{\frac 12}) \nabla \teta^{\frac 12} \cdot \nabla (\teta^{\frac 12}-\eta^{\frac 12}) -\teta^{-1} \Delta \teta (\teta^{\frac 12}-\eta^{\frac 12})^2\Big](\tau,\cdot) \dx \\
&+ \int_\Omega (\tn{T}-\tnT): \nabla (\vu-\tvu) (\tau,\cdot)\dx\\
& = \sum_{j=1}^{10} \mathcal L_j(\tau).
\end{aligned}
$$
Above, we used that the corresponding part of $\mathcal R_1^6(\tau)$ (with $kL$), the corresponding part of $\mathcal R_2(\tau)$ and $\mathcal R_3(\tau)$ give together $\mathcal L_9(\tau)$ and $\mathcal R_1^6(\tau)$ (with $\zeta$), the corresponding part of $\mathcal R_2(\tau)$ and $\mathcal R_4(\tau)$ sum up to zero.

We now estimate all terms in the sum $\sum_{j=1}^{10} \mathcal L_j(\tau)$. We have
$$
\mathcal L_1(\tau) = \int_\Omega  \vr((\vu-\tvu)\cdot \nabla \tvu)\cdot (\tvu-\vu) (\tau,\cdot)\dx \leq \|\nabla\tvu\|_{L^\infty}(\tau) \int_\Omega \vr|\vu-\tvu)|^2 (\tau,\cdot)\dx \leq F_1(\tau) \mathcal E_1(\tau),
$$
where $\mathcal F_1(\tau) \in L^1(0,T)$ provided $\nabla \tvu\in L^1(0, T;L^\infty(\Omega;\R^3))$. Further
$$
\begin{aligned}
\mathcal L_2(\tau) &= 2\mu_0 \int_\Omega \Div((1+ |\tn{D}^d(\tvu)|^2)^{\frac{r-2}{2}}\tn{D}^d(\tvu))\cdot (\tvu-\vu) (\vr \tvr^{-1}-1) (\tau,\cdot)\dx \leq \|J\|_3 \|\tvu-\vu\|_6 \|\tvr-\vr\|_2 \\
& \leq \delta \|\nabla(\tvu-\vu)\|_2^2 + \|J\|_3^2 \|\tvr-\vr\|_2^2 = \delta \|\nabla(\tvu-\vu)\|_2^2 + F_2(\tau) \mathcal E_1(\tau); 
\end{aligned}
$$
here the number $\delta$ denotes sufficiently small number so that the corresponding term can be later transferred to the left-hand side and $F_2\in L^1(0,T)$ provided $ \Div((1+ |\tn{D}^d(\tvu)|^2)^{\frac{r-2}{2}}\tn{D}^d(\tvu)) \in L^2(0,T;L^3(\Omega;\R^3))$. Next
$$
\begin{aligned}
\mathcal L_3(\tau) &= \int_\Omega \nabla \Lambda'(\Div\tvu) \cdot( \tvu- \vu) (\vr \tvr^{-1}-1) (\tau,\cdot)\dx \leq \| \nabla \Lambda'(\Div\tvu)\|_3 \|\tvu-\vu\|_6 \|\tvr-\vr\|_2 \\ 
&\leq \delta \|\nabla(\tvu-\vu)\|_2^2 + \|\nabla \Lambda'(\Div\tvu)\|_3^2 \|\tvr-\vr\|_2^2 \leq \delta \|\nabla(\tvu-\vu)\|_2^2 + F_3(\tau) \mathcal E_1(\tau),
\end{aligned}
$$
where $F_3\in L^1(0,T)$ provided $\nabla \Lambda'(\Div\tvu) \in L^2(0,T;L^3(\Omega;\R^3))$. Next term can be estimated rather similarly,
$$
\mathcal L_4(\tau) = \int_\Omega (\teta-\eta)(\tvu-\vu)\cdot \nabla G'(\teta) (\tau,\cdot)\dx \leq \|\teta-\eta\|_2 \|\tvu-\vu\|_6 \|\nabla G'(\teta)\|_3 \leq \delta \|\nabla(\tvu-\vu)\|_2^2 + F_4(\tau) \mathcal E_2(\tau).
$$
In order to have $F_4 \in L^1(0,T)$, we have to assume that $\nabla G'(\teta)\sim |\nabla \teta| (1+ \teta^{-1}) \in L^2(0,T;L^3(\Omega;\R^3))$. Further 
$$
\mathcal L_5(\tau) = \int_\Omega \Div \tvu (p(\tvr)-p(\vr) -p'(\tvr) (\tvr-\vr)) (\tau,\cdot)\dx \leq F_5(\tau) \|\vr-\tvr\|_2^2 \leq F_5(\tau) \mathcal E_1(\tau),
$$
where due to the boundedness of $\Div \tvu$ the term $F_5\in L^1(0,T)$ trivially. Exactly in the same way we estimate
$$
\mathcal L_6(\tau)  \int_\Omega \Div \tvu (q(\teta)-q(\eta) -q'(\teta) (\teta-\eta)) (\tau,\cdot)\dx \leq F_6(\tau) \|\teta-\eta\|_2^2 \leq F_6(\tau) \mathcal E_2(\tau).
$$
Next
$$
\begin{aligned}
\mathcal L_7(\tau) &= \int_\Omega (1-\vr\tvr^{-1}) \nabla q(\teta)\cdot (\tvu-\vu) (\tau,\cdot)\dx \leq C \|\nabla q(\teta)\|_3 \|\tvu-\vu\|_6\|\tvr-\vr\|_2 \\
&\leq \delta \|\nabla(\tvu-\vu)\|_2^2 + C \|\nabla q(\teta)\|_3^2 \|\tvr-\vr\|_2^2 
\leq  \delta \|\nabla(\tvu-\vu)\|_2^2 + F_7(\tau) \mathcal E_1(\tau),
\end{aligned}
$$
where we need to require that $\nabla q(\teta) \sim |\nabla \teta|(1+ \teta) \in L^2(0,T;L^3(\Omega;\R^3))$. We also used that $\tvr^{-1}$ is bounded.  Similarly also
$$
\begin{aligned}
\mathcal L_8(\tau) &= \int_\Omega (\vr\tvr^{-1}-1) \Div \tnT \cdot (\tvu-\vu)(\tau,\cdot)\dx \leq C \|\Div \tnT\|_3 \|\tvu-\vu\|_6\|\tvr-\vr\|_2 \\
&\leq \delta \|\nabla(\tvu-\vu)\|_2^2 + C \|\Div \tnT\|_3^2 \|\tvr-\vr\|_2^2 
\leq  \delta \|\nabla(\tvu-\vu)\|_2^2 + F_8(\tau) \mathcal E_1(\tau),
\end{aligned}
$$
where in order to have $F_8 \in L^1(0,T)$ we need that $\Div \tnT \in  L^2(0,T;L^3(\Omega;\R^3))$. The term $\mathcal L_9(\tau)$ is more complex and we divide it therefore into two parts. First we have
$$
\begin{aligned}
\mathcal L_9^1(\tau) & = \varepsilon kL \int_\Omega  4 \teta^{-\frac 12} (\teta^{\frac 12}-\eta^{\frac 12}) \nabla \teta^{\frac 12} \cdot \nabla (\teta^{\frac 12}-\eta^{\frac 12}) (\tau,\cdot)\dx \\
& \leq C \|\teta^{-1}\|_\infty \|\nabla\teta\|_\infty \|\teta^{\frac 12}-\eta^{\frac 12}\|_2 \|\nabla (\teta^{\frac 12}-\eta^{\frac 12})\|_2 \\
& \leq \delta \|\nabla (\teta^{\frac 12}-\eta^{\frac 12})\|_2^2 + \|\teta^{-1}\|_\infty^2 \|\nabla\teta\|_\infty^2 \|\teta^{\frac 12}-\eta^{\frac 12}\|_2^2 \leq \delta \|\nabla (\teta^{\frac 12}-\eta^{\frac 12})\|_2^2 + F_9^1(\tau) \mathcal E_2(\tau),
\end{aligned}
$$
provided $\nabla\teta\in L^2(0,T;L^\infty(\Omega;\R^3))$; above we used that $\teta^{-1}$ is bounded as well as
$$
|\teta^{\frac 12}-\eta^{\frac 12}|^2 = \Big(\frac{(\teta^{\frac 12}-\eta^{\frac 12})(\teta^{\frac 12}+\eta^{\frac 12})}{\teta^{\frac 12}+\eta^{\frac 12}}\Big)^2 \leq C|\teta-\eta|^2,
$$
as the denominator is bounded away from zero.
Slightly differently we deal with the part 
$$
\begin{aligned}
\mathcal L_9^2(\tau) &= \varepsilon kL \int_\Omega \teta^{-1} \Delta \teta (\teta^{\frac 12}-\eta^{\frac 12})^2(\tau,\cdot) \dx  = \\
& = \varepsilon kL \int_\Omega \Big[\teta^{-2} |\nabla \teta|^2 (\teta^{\frac 12}-\eta^{\frac 12})^2 - 2\teta^{-1} \nabla \teta\cdot \nabla (\teta^{\frac 12}-\eta^{\frac 12}) (\teta^{\frac 12}-\eta^{\frac 12})(\tau,\cdot) \dx \\
& \leq C\|\teta^{-1}\|_\infty^2 \|\nabla\teta\|_\infty^2 \| \teta^{\frac 12}-\eta^{\frac 12}\|_2^2 + C \|\teta^{-1}\|_\infty \|\nabla\teta\|_\infty |\teta^{\frac 12}-\eta^{\frac 12}\|_2 \|\nabla (\teta^{\frac 12}-\eta^{\frac 12})\|_2 \\
& \leq \delta \|\nabla (\teta^{\frac 12}-\eta^{\frac 12})\|_2^2 +  F_9^2(\tau) \|\teta-\eta\|_2^2 \leq \delta \|\nabla (\teta^{\frac 12}-\eta^{\frac 12})\|_2^2 +  F_9^2(\tau) \mathcal E_2(\tau),
\end{aligned}
$$
where we used similar estimate as above. In order to have $F_9^2\in L^1(0,T)$, we need again $\nabla \teta \in L^2(0,T;L^\infty(\Omega;\R^3))$. Finally, we have
$$
\mathcal L_{10}(\tau) =  \int_\Omega (\tn{T}-\tnT): \nabla (\vu-\tvu) (\tau,\cdot)\dx \leq \|\nabla (\tvu-\vu)\|_2 \|\tnT-\tn{T}\|_2 \leq \delta \|\nabla (\tvu-\vu)\|_2^2 + C\|\tnT-\tn{T}\|_2^2.
$$
We therefore obtain, due to \eqref{5.13}
\begin{equation} \label{5.14}
\begin{aligned}
&\mathcal E_1[(\vr,\vu),(\tvr,\tvu)](t) + \mathcal E_2[\eta,\teta](t) +  \frac 12 \int_0^t\int_\Omega (\Lambda (\Div\vu) - \Lambda (\Div \tvu) - \Lambda'(\Div\tvu) (\Div \vu-\Div \tvu))\dx\dtau \\
& + \frac 12 \int_0^t \int_\Omega 2\mu_0\Big[(1+ |\tn{D}^d(\vu)|^2)^{\frac{r-2}{2}} \tn{D}^d(\vu)- (1+ |\tn{D}^d(\tvu)|^2)^{\frac{r-2}{2}}\tn{D}^d(\tvu)\Big]:(\tn{D}^d(\vu)-\tn{D}^d(\tvu)) \dx\dtau \\
& + \varepsilon \int_0^t \int_\Omega \Big(2kL \big|\nabla (\eta^{\frac 12}-\teta^{\frac 12})\big|^2 + \zeta |\nabla (\eta-\teta)|^2\Big)\dx\dtau  \\
&  \leq \int_0^t F(\tau) \big(\mathcal E_1[(\vr,\vu),(\tvr,\tvu)](\tau) + \mathcal E_2[\eta,\teta](\tau)\big) \dtau + \int_0^t \|\tnT-\tn{T}\|_2^2(\tau)\dtau,
\end{aligned}
\end{equation}
where under the assumptions above, $F \in L^1(0,T)$.
However, we have no control of the term $\|\tnT-\tn{T}\|_2^2$ yet and we need to establish it. To this aim, we use equations \eqref{04a} both for $\tn{T}$ in the weak form and for $\tnT$ in the strong form. This is seemingly straightforward, however, we must be careful. A straightforward calculation would be to take the weak formulation for $\tn{T}$ and use as a test function $\tn{T}-\tnT$ and then do the same for for $\tnT$ and subtract it. This, however, contains a few problems with integrability of the term $\int_0^t \int_\Omega \Div (\tn{T}\vu)\cdot \tn{T}$ as well as with integrability of the time derivative of $\tn{T}$. We therefore carefully explain the procedure in detail.

First, we have 
\begin{equation} \label{T1}
\begin{aligned}
&\frac 12 \int_\Omega |\tn{T}|^2 (t,\cdot)\dx + \varepsilon \int_0^t \int_\Omega |\nabla \tn{T}|^2 \dx\dtau + \frac 1{2\lambda}\int_0^t \int_\Omega |\tn{T}|^2 \dtau\dx \\
&\leq \frac 12 \int_\Omega |\tn{T}_0|^2 \dx -\frac 12 \int_\Omega \Div \vu |\tn{T}|^2 \dx\dtau + 2 \int_0^t \int_\Omega (\nabla \vu \tn{T}):\tn{T}\dx\dtau + \frac{k}{2\lambda} \int_\Omega \eta {\rm Tr}\, \tn{T} \dx\dtau.
\end{aligned}
\end{equation}
We namely work with a form of equality \eqref{energy_est_T} and it is enough to perform all limit passages for this equality (as a result, we, however, get the inequality). The same, now equality, can be obtained for the regular solutions since by our assumptions, they are sufficiently regular (note that here we use that $\pder{\tnT}{t}$ must be integrable)
\begin{equation} \label{T2}
\begin{aligned}
&\frac 12 \int_\Omega |\tnT|^2 (t,\cdot)\dx + \varepsilon \int_0^t \int_\Omega |\nabla \tnT|^2 \dx\dtau + \frac 1{2\lambda}\int_0^t \int_\Omega |\tnT|^2 \dtau\dx \\
&= \frac 12 \int_\Omega |\tn{T}_0|^2 \dx -\frac 12 \int_\Omega \Div \tvu |\tnT|^2 \dx\dtau + 2 \int_0^t \int_\Omega (\nabla \tvu \tnT):\tnT\dx\dtau + \frac{k}{2\lambda} \int_\Omega \teta {\rm Tr}\, \tnT \dx\dtau.
\end{aligned}
\end{equation} 
Next, we can clearly use as test function in the weak formulation for $\tn{T}$ the function $\tnT$ (see \eqref{weak_extra_stress}) and get
\begin{equation} \label{T3}
\begin{aligned}
&\int_0^t\int_\Omega \tn{T}:\pder{\tnT}{t} \dx\dtau  + \int_0^t \int_\Omega \vu \tn{T} :: \nabla \tnT \dx\dtau - \varepsilon \int_0^t \int_\Omega \nabla \tn{T}::\nabla\tnT  \dx\dtau - \frac 1{2\lambda}\int_0^t \int_\Omega \tn{T}:\tnT \dtau\dx \\
&= \int_\Omega \tn{T}:\tnT (t,\cdot) \dx -\int_\Omega |\tn{T}_0|^2 \dx - \int_0^t \int_\Omega (\nabla \vu \tn{T} +\tn{T}\nabla^T \vu):\tnT\dx\dtau - \frac{k}{2\lambda} \int_\Omega \eta {\rm Tr}\, \tnT \dx\dtau.
\end{aligned}
\end{equation} 
Finally, we may multiply the equation for $\tnT$ by $\tn{T}$ and integrate by parts in space. We get
\begin{equation} \label{T4}
\begin{aligned}
&\int_0^t \int_\Omega \pder{\tnT}{t}:\tn{T} \dx\dtau + \int_0^t \int_\Omega {\rm Div}\,(\tvu\tnT):\tn{T}\dx\dtau - \int_0^t \int_\Omega (\nabla \tvu \tnT +\tn{T} \nabla^T \tnT):\tn{T} \dx\dtau \\
&= -\varepsilon \int_0^t \int_\Omega \nabla \tnT::\nabla \tn{T} \dx\dtau + \frac{k}{2\lambda} \int_\Omega \teta {\rm Tr}\, \tn{T}\dx\dtau -\frac{1}{2\lambda} \int_0^t \int_\Omega \tnT:\tn{T}\dx\dtau.
\end{aligned}
\end{equation}
We now take \eqref{T1} + \eqref{T2} + \eqref{T3} $-$ \eqref{T4}. Thus we conclude, after straightforward computations, that
$$
\begin{aligned}
&\frac 12 \int_\Omega |\tn{T}-\tnT|^2(t,\cdot)\dx + \varepsilon \int_0^t \int_\Omega |\nabla (\tn{T}-\tnT)|^2 \dx\dtau + \frac 1{2\lambda} \int_0^t \int_\Omega |\tn{T}-\tnT|^2\dx\dtau \\
&\leq \int_0^t \int_\Omega (\eta-\teta)({\rm Tr}\, \tn{T} -{\rm Tr}\, \tnT) \dx\dtau \\
&- \int_0^t \int_\Omega \Big(\frac 12 \Div \vu |\tn{T}|^2 + \frac 12\Div \tvu |\tnT|^2 + \vu \tn{T}::\nabla \tnT -{\rm Div}\, (\tvu\tnT):\tn{T}\Big)\dx\dtau  \\
& + 2\int_0^t\int_\Omega \Big((\nabla \vu \tn{T}):\tn{T} +  (\nabla\tvu \tnT):\tnT - (\nabla \vu \tn{T}):\tnT - (\nabla \tvu \tnT): \tn{T} \Big)\dx\dtau.
\end{aligned}
$$ 
Next we compute
$$
\begin{aligned}
&- \int_0^t \int_\Omega \Big(\frac 12 \Div \vu |\tn{T}|^2 + \frac 12\Div \tvu |\tnT|^2 + \vu \tn{T}::\nabla \tnT -{\rm Div}\, (\tvu\tnT):\tn{T}\Big)\dx\dtau  \\
& =\int_0^t \int_\Omega \Big(-\frac 12 \Div \vu |\tn{T}|^2 + \Div \vu \tn{T}:\tnT + (\tvu\cdot \nabla \tnT): \tnT + (\vu\cdot \nabla \tn{T}): \tnT  - (\tvu \cdot \nabla \tn{T}):\tnT\Big)\dx\dtau  \\
& =\int_0^t \int_\Omega \Big(-\frac 12 \Div \vu |\tn{T}-\tnT|^2 + \frac 12 \Div \vu |\tnT|^2 + (\tvu\cdot \nabla \tnT): \tnT + (\vu\cdot \nabla \tn{T}): \tnT  - (\tvu \cdot \nabla \tn{T}):\tnT\Big)\dx\dtau \\
& =\int_0^t \int_\Omega \Big(-\frac 12 \Div \vu |\tn{T}-\tnT|^2  + \big((\tvu-\vu)\cdot \nabla \tnT\big):\tnT -\big((\tvu-\vu)\cdot \nabla \tn{T}\big):\tnT\Big)\dx\dtau \\
& =\int_0^t \int_\Omega \Big(-\frac 12 \Div \vu |\tn{T}-\tnT|^2  + \big((\tvu-\vu)\cdot \nabla (\tnT-\tn{T})\big):\tnT\Big)\dx\dtau.
\end{aligned}
$$
Similarly
$$
\begin{aligned}
&\int_0^t\int_\Omega \Big((\nabla \vu \tn{T}):\tn{T} +  (\nabla\tvu \tnT):\tnT - (\nabla \vu \tn{T}):\tnT - (\nabla \tvu \tnT) :\tn{T} \Big)\dx\dtau \\
& =\int_0^t\int_\Omega \Big((\nabla \tvu (\tn{T}-\tnT)):(\tn{T}-\tnT) + \big(\nabla(\vu-\tvu)\tn{T}\big):\tn{T} + \big(\nabla(\tvu-\vu)\tn{T}\big):\tnT \Big)\dx\dtau \\ 
& =\int_0^t\int_\Omega \Big((\nabla \tvu (\tn{T}-\tnT)):(\tn{T}-\tnT) + \big(\nabla(\vu-\tvu)\tn{T}\big):(\tn{T}-\tnT)\Big)\dx\dtau.
\end{aligned}
$$
Thus we have
\begin{equation}\label{T5}
\begin{aligned}
&\frac 12 \int_\Omega |\tn{T}-\tnT|^2(t,\cdot)\dx + \varepsilon \int_0^t \int_\Omega |\nabla (\tn{T}-\tnT)|^2 \dx\dtau + \frac 1{2\lambda} \int_0^t \int_\Omega |\tn{T}-\tnT|^2\dx\dtau \\
&\leq \int_0^t \int_\Omega (\eta-\teta)({\rm Tr}\, \tn{T} -{\rm Tr}\, \tnT) \dx\dtau \\
& +\int_0^t \int_\Omega \Big(-\frac 12 \Div \vu |\tn{T}-\tnT|^2  + \big((\tvu-\vu)\cdot \nabla (\tnT-\tn{T})\big):\tnT\Big)\dx\dtau\\
& +2\int_0^t\int_\Omega \Big((\nabla \tvu (\tn{T}-\tnT)):(\tn{T}-\tnT) + \big(\nabla(\vu-\tvu)\tn{T}\big):(\tn{T}-\tnT)\Big)\dx\dtau\\
&= \sum_{i=1}^5 \int_0^t \mathcal M_i(\tau)\dtau.
\end{aligned}
\end{equation}
We now estimate each term separately. We have
$$
\begin{aligned}
\mathcal M_1(\tau) &= \int_\Omega (\eta-\teta)({\rm Tr}\, \tn{T} -{\rm Tr}\, \tnT) (\tau,\cdot) \dx  \\
& \leq \|\eta-\teta\|_2(\tau) \|\tn{T}-\tnT\|_2(\tau) \leq   \mathcal E_2[\eta,\teta](\tau) + C \|\tn{T}-\tnT\|_2^2(\tau).
\end{aligned}
$$
Next we have
$$
\begin{aligned}
\mathcal M_2(\tau) &= \int_\Omega -\frac 12 \Div \vu |\tn{T}-\tnT|^2 (\tau,\cdot)\dx  \\
& \leq \|\Div \vu\|_\infty (\tau)\|\tn{T}-\tnT\|_2^2(\tau) \leq C \|\tn{T}-\tnT\|_2^2(\tau),
\end{aligned}
$$
as the velocity divergence is uniformly bounded. The next term already requires some better regularity of the strong solution,
$$
\begin{aligned}
\mathcal M_3(\tau) &= \int_\Omega \big((\tvu-\vu)\cdot \nabla (\tnT-\tn{T})\big):\tnT (\tau,\cdot) \dx\\
       & \leq \|\vu-\tvu\|_2(\tau) \|\nabla(\tn{T}-\tnT)\|_2(\tau) \|\tnT\|_\infty(\tau) \\
			&\leq \delta \|\nabla(\tn{T}-\tnT)\|_2^2(\tau) + C(T)\|\tnT\|_\infty^2(\tau) \mathcal E_1[(\vr,\vu),(\tvr,\tvu)](\tau),
			\end{aligned}
$$			
as the density is bounded away from zero for any fixed positive time. Next
$$
\begin{aligned}
\mathcal M_4(\tau) &= 2\int_\Omega \Big((\nabla \tvu (\tn{T}-\tnT)):(\tn{T}-\tnT) (\tau,\cdot) \dx  \\
& \leq 2\|\nabla \tvu\|_\infty(\tau) \|\tn{T}-\tnT\|_2^2(\tau).
\end{aligned}
$$
The last term, however, will require a certain higher integrability of the extra stress tensor $\tn{T}$. We first estimate the term $M_5$ assuming that the corresponding higher integrability of the extra stress tensor is available and then show that under certain assumptions on the initial condition for $\tn{T}$ this estimate is in fact available. 
We have ($a$ is a small positive number, sufficiently close to zero; in fat, $a<1$ is sufficient)
$$
\begin{aligned}
\mathcal M_5(\tau) &= 2\int_\Omega \big(\nabla(\vu-\tvu)\tn{T}\big):(\tn{T}-\tnT)(\tau,\cdot)\dx  \\
& \leq 2 \|\tn{T}\|_{3+a}(\tau) \|\nabla(\vu-\tvu)\|_2(\tau) \|\tn{T}-\tnT\|_{\frac{2(3+a)}{1+a}}(\tau) \\
& \leq 2 \|\tn{T}\|_{3+a}(\tau) \|\nabla(\vu-\tvu)\|_2(\tau) \|\tn{T}-\tnT\|_2^{\frac{a}{3+a}}(\tau)\|\tn{T}-\tnT\|_{1,2}^{\frac {3}{3+a}}(\tau)\\
& \leq \delta \|\nabla(\vu-\tvu)\|_2^2(\tau) + \delta \|\tn{T}-\tnT\|_{1,2}^2 (\tau)+ C(\delta) \|\tn{T}\|_{3+a}^{\frac{2(3+a)}{a}}(\tau) \|\tn{T}-\tnT\|_{2}^2(\tau).
\end{aligned}
$$
In order to conclude by summing up \eqref{5.14} and \eqref{T5} and employing estimates above, we need to know that $\|\tn{T}\|_{3+a}^{\frac{2(3+a)}{a}}$ is integrable over the time interval. We show now that, assuming $\|\tn{T}_0\|_{3+a}$ bounded and $a$ sufficiently small positive, then this term is in fact bounded.

To this aim we would like to use the second part of Lemma \ref{parabolic_Neumann} on equation \eqref{04a}. However, this is not possible to do directly, as we do not have most of the terms under divergence. We therefore first solve an auxiliary problem
$$
\Div (g_{ij}) = \Big[(\nabla \vu \tn{T} + \tn{T} \nabla^T \vu) + \frac{k}{2\lambda}\eta \tn{I} -\frac {1}{2\lambda}\tn{T}\Big]_{ij} - \int_\Omega \Big[(\nabla \vu \tn{T} + \tn{T} \nabla^T \vu) + \frac{k}{2\lambda}\eta \tn{I} -\frac {1}{2\lambda}\tn{T}\Big]_{ij} \dx
$$
in $\Omega$ with homogeneous Dirichlet boundary condition for $g_{ij}$, $i$, $j=1,2,3$. It is easy to see that the most restrictive term is the quadratic one and therefore get on almost every time level that
$$
\|\nabla(g_{ij})\|_q \leq C(\|\nabla \vu \tn{T}\|_q +1).
$$
Our correct choice of $q$ is $\frac{3(3+a)}{6+a}$ since then $W^{1,q}(\Omega) \hookrightarrow L^{3+a}(\Omega)$. We therefore get 
$$
\|\nabla \vu \tn{T}\|_{\frac{3(3+a)}{6+a}} \leq \|\nabla \vu\|_{\frac 52} \|\tn{T}\|_{\frac{15(3+a)}{12-a}}.
$$
Note that $\frac{15(3+a)}{12-a} \sim \frac{15}{4} <4$ for $a$ sufficiently small, therefore  the right-hand side above is surely integrable over the time interval $(0,T)$ in a power larger than 1. 

We may now apply the second part of Lemma \ref{parabolic_Neumann} (note that the term ${\rm Div}\, (\vu\tn{T})$ behaves exactly as $g_{ij}$) and we conclude that $\tn{T}\in L^\infty(0,T;L^{3+a}(\Omega))$. 

\begin{remark}
Note also that for $r$ sufficiently large ($r\geq 10$) we can get corresponding estimate without any further assumption on the higher integrability of the initial condition of the extra stress tensor. Recall that in fact, we need to estimate $\int_0^T\|\tn{T}\|_{3+a}^{\frac{2(3+a)}{a}} \dtau$. As can be seen from the above computations, it corresponds to the fact that
$$
\int_0^T \|\nabla \vu\tn{T}\|_{\frac{3(3+a)}{6+a}}^{\frac{2(3+a)}{a}} \dtau <\infty.
$$
Computing
$$
\|\nabla \vu\tn{T}\|_{\frac{3(3+a)}{6+a}}^{\frac{2(3+a)}{a}} \leq \|\tn{T}\|_{2}^{\frac{2(3+a)}{a}} \|\nabla \vu\|_{\frac{6(3+a)}{3-a}}^{\frac{2(3+a)}{a}},
$$
it is enough to check when $\frac{6(3+a)}{3-a} = \frac{2(3+a)}{a}$ which happens for $a=\frac 34$ and thus $r_{\rm min}=\frac{6(3+a)}{3-a}|_{a=\frac 34} = 10$. Indeed, there is a space to improve the result by treating the estimates more carefully, but we do not intend to do so since the improvement is rather marginal with respect to the main result. 
\end{remark}

We can now sum \eqref{5.14} and \eqref{T5} and conclude that under the assumptions of Theorem \ref{main_2} we get
\begin{equation} \label{5.15}
\begin{aligned}
&\mathcal E_1[(\vr,\vu),(\tvr,\tvu)](t) + \mathcal E_2[\eta,\teta](t) + \frac 12 \int_\Omega |\tn{T}-\tnT|^2(t,\cdot)\dx \\
&+ \frac 12 \int_0^t\int_\Omega (\Lambda (\Div\vu) - \Lambda (\Div \tvu) - \Lambda'(\Div\tvu) (\Div \vu-\Div \tvu))\dx\dtau \\
& + \frac 12 \int_0^t \int_\Omega 2\mu_0\Big[(1+ |\tn{D}^d(\vu)|^2)^{\frac{r-2}{2}} \tn{D}^d(\vu)- (1+ |\tn{D}^d(\tvu)|^2)^{\frac{r-2}{2}}\tn{D}^d(\tvu)\Big]:(\tn{D}^d(\vu)-\tn{D}^d(\tvu)) \dx\dtau \\
& + \varepsilon \int_0^t \int_\Omega \Big(2kL \big|\nabla (\eta^{\frac 12}-\teta^{\frac 12})\big|^2 + \zeta |\nabla (\eta-\teta)|^2\Big)\dx\dtau + \varepsilon \int_0^t \int_\Omega |\tn{T} -\tnT|^2 \dx\dtau \\
& \int_0^t F_1(\tau) \big(\mathcal E_1[(\vr,\vu),(\tvr,\tvu)](\tau) + \mathcal E_2[\eta,\teta](\tau) + \|\tn{T}-\tnT\|_2^2(\tau) \big) \dtau,
\end{aligned}
\end{equation}
where $F_1(\tau)$ is integrable over $(0,T)$ which concludes the proof of Theorem \ref{main_2}.

\bigskip

{\bf Acknowledgment:} 
The work of Y.L. was partially supported... .The workj of M.P. was partially supported by the Czech Science Foundation (GA\v{C}R), project No. 25-16592S.

\end{document}